\documentclass[11pt]{amsart}
\usepackage{amsmath,amsthm,amscd,amssymb, color}
\usepackage{latexsym}
\usepackage[colorlinks, citecolor = blue,backref]{hyperref}
\usepackage[alphabetic]{amsrefs}
\usepackage[capitalize, nameinlink]{cleveref}
\usepackage{enumerate}
\usepackage[margin=1.333in]{geometry}
\usepackage{pgfplots}
\usepackage{tkz-graph}
\usepackage[font=footnotesize]{caption}

\newcommand{\R}{\mathbb R}

\newcommand{\Q}{\mathbb Q}
\newcommand{\C}{\mathbb C}
\newcommand{\Z}{\mathbb Z}

\renewcommand{\phi}{\varphi}
\newcommand{\eps}{\varepsilon}

\newcommand{\frako}{\mathfrak o}
\newcommand{\calO}{\mathcal O}

\newcommand{\calF}{\mathcal F}
\newcommand{\calI}{\mathcal I}
\newcommand{\calJ}{\mathcal J}
\newcommand{\A}{\mathbb A}
\newcommand{\bs}{\backslash}
\newcommand{\bmx}{\left( \begin{matrix}}
\newcommand{\emx}{\end{matrix} \right)}
\newcommand{\Cl}{\mathrm{Cl}}
\newcommand{\Ram}{\mathrm{Ram}}

\newcommand{\new}{\mathrm{new}}
\newcommand{\odd}{\mathrm{odd}}

\renewcommand{\mod}{\bmod}
\newcommand{\Sq}{\mathrm{Sq}}
\newcommand{\Orb}{\mathrm{Orb}}
\newcommand{\Gal}{\mathrm{Gal}}

\usepackage{bbm}
\newcommand{\one}{\mathbbm{1}}

\newcommand{\norm}[1]{\left\lVert#1\right\rVert_\infty}
\newcommand{\leg}{\overwithdelims ()}

\DeclareMathOperator{\GL}{GL}

\DeclareMathOperator{\tr}{tr} 
 
\DeclareMathOperator{\Hom}{Hom} 

\newtheorem{lem}{Lemma}
\numberwithin{lem}{section}
\newtheorem{prop}[lem]{Proposition}
\newtheorem{thm}[lem]{Theorem}

\newtheorem{cor}[lem]{Corollary}
\newtheorem{conj}[lem]{Conjecture}
\crefname{figure}{Figure}{Figures}
\crefname{conj}{Conjecture}{Conjectures}

\theoremstyle{remark}
\newtheorem{rem}[lem]{Remark}

\theoremstyle{definition}
\newtheorem{example}[lem]{Example}

\numberwithin{equation}{section}

\begin{document}

\title{Zeroes of quaternionic modular forms and central $L$-values}
\author{Kimball Martin}
\email{kimball.martin@ou.edu}
\address{Department of Mathematics, University of Oklahoma, Norman, OK 73019 USA}

\author{Jordan Wiebe}
\email{jwiebe@ou.edu}

\date{\today}

\maketitle

\begin{abstract}
Values of quaternionic modular forms are related to twisted
central $L$-values via periods and a theorem of Waldspurger.  
In particular, certain twisted $L$-values must be non-vanishing for
forms with no zeroes.  Here we study, theoretically and
computationally, zeroes of definite quaternionic
modular forms of trivial weight.
Local sign conditions force certain forms to have trivial
zeroes, but we conjecture that almost all forms have no nontrivial
zeroes.  In particular, almost all forms with appropriate local signs
should have no zeroes.  
We show these conjectures follow from a conjecture on the
average number of Galois orbits, and give applications  
to (non)vanishing of $L$-values.
\end{abstract}


\section{Introduction}


\subsection{Motivation}
Let $S_k(N)$ denote the space of weight $k$
elliptic cusp forms of level $\Gamma_0(N)$, and let
$S^\new_k(N)$ denote the subspace of newforms.
Let $B = B_N$ a definite quaternion algebra of discriminant $N$,
and $\calO = \calO_N$ a maximal order of $B$.
Necessarily, $N$ is a squarefree
product of an odd number of primes.  Let $\Cl(\calO)$ be the
set of right $\calO$-ideal classes in $B$, which is a finite set.

The space of trivial weight modular forms of level $\calO$ on $B$
is simply the space of functions $M(\calO) = \{ \phi : \Cl(\calO) \to \C \}$.
The subspace of cusp forms $S(\calO)$ can be defined as the 
codimension 1 subspace of forms orthogonal to the constant functions,
with respect to the Petersson inner product.
There is a well-known action of Hecke operators, and the
Eichler--Shimizu--Jacquet--Langlands correspondence gives a
Hecke isomorphism of $S(\calO)$ with $S^\new_2(N)$.
In particular, there is a 1-1 correspondence of Hecke eigenforms 
$\phi \in S(\calO)$ (modulo scalars) with newforms $f \in S_2(N)$.

A formula of Waldspurger \cite{wald} relates
certain twisted $L$-values to certain periods on $B$.   Namely,
if  $\phi$ and $f$ correspond as above, 
central $L$-values associated to certain twists of $f$ are nonzero if and 
only if certain linear combinations of values of $\phi$ called periods are nonzero.
Thus a natural approach to studying the vanishing of 
$L$-values is to study both (i) values of quaternionic
modular (eigen)forms $\phi$, and (ii) periods on $B$.
While periods have often been used to study $L$-values (e.g., see \cite{jacquet-chen},
\cite{michel-venkatesh}), we understand relatively little about the values of
quaternionic eigenforms.

The present paper represents one step in an attempt to
understand (i) and (ii), but especially values of quaternionic
modular forms. Specifically, we focus on the question of how often
values of quaternionic eigenforms are zero, and what this implies
about non-vanishing $L$-values.  In particular, let $K=\Q(\sqrt{-D})$ be the
imaginary quadratic field of discriminant $-D$ and $\eta_D(n) = {-D \leg n}$
the associated quadratic Dirichlet character.  We show that if $\phi$ is zero-free
and either (1) $K$ has class number 1 with $\eta_D(p) \ne 1$ for all $p | N$,
or (2) $K$ has one class per genus and $D | N$,
then $L(\frac 12, f) L(\frac 12, f \otimes \eta_D) \ne 0$ (\cref{thm:16}). 

Moreover, we conjecture that $\phi$ is zero-free 100\% of the time that certain necessary 
conditions on the Atkin--Lehner signs $w_p$ of $f$ are satisfied---e.g., if 
$w_p = -1$ for all $p | N$ 
(see \cref{conj11}).  Assuming these sign conditions hold, \cref{thm:nozeroes}
implies that $\phi$ is zero-free if the Galois conjugates of $f$ span the 
Atkin--Lehner eigenspace of $S_2^\new(N)$ for the signs $w_p$---something
that was conjectured to hold 100\% of the time in \cite{me:maeda}.
While we do not have other simple criteria to guarantee that $\phi$ is zero-free, it is
something that one can check computationally.  

We note that theorems which guarantee the nonvanishing of 
\emph{specific} central $L$-values typically rely on special arithmetic information like 
the properties of certain elliptic curves or the existence of an Eisenstein series
congruent to $f$.  Our emphasis is that knowledge of the values of quaternionic 
modular forms offers a different method to obtain nonvanishing of explicit twists.
This idea was also exploited in \cite{me:cong}, which used information on values of
quaternionic modular forms to construct Eisenstein congruences, and then apply this
to $L$-values.  In fact, if $f$ satisfies
Eisenstein congruence, then under some conditions one can conclude
 $\phi$ is zero-free (see \cite[Remark 2.4]{me:cong}).

Next we describe the contents in more detail.
We remark that much of what we say extends to Eichler and even more general 
orders.  However, for simplicity we restrict to maximal orders for the 
introduction and the bulk of the paper.

\subsection{Zeroes of quaternionic modular forms}

First, we note that $S(\calO)$ breaks up into Hecke-invariant
subspaces which are the analogue of Atkin--Lehner eigenspaces for
$S_2^\new(N)$.  Namely, for each $p | N$, $T_p$ acts as an involution
on $S(\calO)$, so each eigenform $\phi \in S(\calO)$ has
a collection of local signs $\eps_p(\phi) = \pm 1$ for $p | N$, which are the 
$T_p$-eigenvalues of $\phi$.  These signs are the opposite of the
Atkin--Lehner signs $w_p$.

In \cite{me:cong2}, we observed that the local signs of $\phi$ 
force $\phi$ to be 0 on certain ideal classes.  We call such zeroes
trivial zeroes, and study them in \cref{sec:triv0}.
In particular, if the product of all local signs is $-1$, which corresponds to
$\phi$ having root number $-1$, then one necessarily has trivial zeroes.
We estimate the number of trivial zeroes, and give a precise formula
in the case of prime level.  If all signs are $+1$,
then there are no trivial zeroes.  If $N$ is prime, so that there is only one
local sign, we thus have a complete answer to when and how many
trivial zeroes there are for any eigenform.
When $N$ is not prime, the situation is more complicated but we give
criteria for sign patterns to induce trivial zeroes.
See \cref{prop36} for necessary and sufficient criteria when $N$ is odd.
 
Using a combination of numerics and heuristics, we make the
following conjectures.
Let $\Sq_r$ denote the positive squarefree
integers with exactly $r \ge 1$ prime factors, and 
$\Sq_r(X) = \{ N \in \Sq_r : N < X \}$.  For each $N \in \Sq_\odd : = \bigcup_{r \, \odd} \Sq_r$, we choose a maximal order $\calO_N$ in the definite
quaternion algebra $B_N$ of discriminant $N$.  For each $N$,
there are a finite
number of choices for $\calO_N$ up to isomorphism, but this
choice does not affect the values of the eigenforms in $S(\calO_N)$
(see \cref{lem1}).

\begin{conj}\label{conj11} Fix $r$ odd.  Consider the collection of all eigenforms 
(modulo scalars) which lie in some $S(\calO_N)$ with $N \in \Sq_r$, partially 
ordered by level. Then we have:

\begin{enumerate}[(i)]
\item $100\%$ of the zeroes of these eigenforms are trivial zeroes; and 

\item $100\%$ of these eigenforms have no trivial zeroes.
\end{enumerate}
\end{conj}

Note $r=1$ here corresponds to restricting to prime levels.  One can 
also formulate a version without fixing $r$, i.e., for $N \in \Sq_\odd$.

For a given eigenform $\phi$, its Hecke eigenvalues generate
a number field $K_\phi$, which equals the rationality field $K_f$ of the
associated newform $f \in S_2(N)$, i.e., the number field generated by the
Fourier coefficients of $f$.  Define the \emph{(rationality) degree} of the
$\phi$ (and also $f$) to be the degree $[K_\phi : \Q]=[K_f:\Q]$ of the rationality field.

\begin{conj}[Rough form] \label{conj12} Quaternionic eigenforms are
more likely to have nontrivial zeroes if they have smaller degree.
\end{conj}

See \cref{conj:nnontriv,conj:deg1}, and \cref{rem:deg2} for more precise
forms of \cref{conj11}(i) and \cref{conj12}.  Our heuristics also lead
us to briefly consider the distribution of values of quaternionic eigenforms
in \cref{sec:vals}.

As evidence for \cref{conj12}, we show that eigenforms with large
degree cannot have too many zeroes (\cref{cor43}).
The idea of the proof is very simple---it uses the fact that
$\phi$ and its Galois conjugates have zeroes at the same locations,
and they span a $d$-dimensional subspace of $S(\calO_N)$.
Then one uses the description of the subspaces of $S(\calO_N)$ with
fixed local signs.  
This result also provides evidence for \cref{conj11}, under some hypotheses 
about the sizes of Galois orbits.
Namely, in \cite{me:maeda}, the first author made the following conjecture.

\begin{conj} \label{gal-conj}   Fix $k \in 2\Z_{>0}$ and $r \in \Z_{>0}$.
The average number of Galois orbits of $S_k^\new(N)$ for $N \in \Sq_r$
is $2^r$, i.e., 
\[ \lim_{X \to \infty} \frac {\sum_{N \in \Sq_r(X)} \# \Orb(S_k^\new(N))}
{\# \Sq_r(X) } = 2^r. \]
Here $\# \Orb(S_k^\new(N))$ denotes the number of Galois orbits of newforms
in $S_k(N)$.
\end{conj}

Then the following is a simple consequence of \cref{thm:nozeroes}.

\begin{thm} \label{thm14}
\cref{gal-conj} (for $k=2$) implies \cref{conj11}(ii).
\end{thm}

\subsection{Nonvanishing of $L$-values}
A well-known theorem
of Waldspurger \cite{wald} relates twisted central $L$-values
to periods on quaternion algebras.  In special situations, we can
use the nonvanishing of values of quaternionic modular forms to
deduce nonvanishing of certain periods, and thus of certain $L$-values.

Let $K=\Q(\sqrt{-D})$ be the imaginary quadratic field of discriminant
$-D$.  Denote by $\Cl(K)$ its ideal class group, $\eta_{D}(n) = {-D \leg n}$,
and $\phi_K$ the base change 
of an eigenform $\phi$ to $K$.
For a quaternionic or elliptic modular eigenform $\phi$ and a character $\chi$ defined
over the same field, let $L(s, \phi \otimes \chi)$
denote the twisted automorphic $L$-function normalized so that $s= \frac 12$
is the central point.  In particular, if $\chi$ is the trivial character of $\Cl(K)$,
then $L(s, \phi_K \otimes \chi) = L(s, \phi_K)= L(s, \phi) L(s, \phi \otimes \eta_{D})$.

Let $B=B_N$.  Then $K$ embeds in $B$ if and only if $\eta_{D}(p) \ne 1$
for each $p | N$.  We may choose a maximal order $\calO_N$ of $B$
and an embedding of $K$ into $B$ such that
$\frako_K$, the ring of integers of $K$, embeds into in $\calO_N$.  
This leads to an ideal class map of sets:
$\Cl(K) \to \Cl(\calO_N)$.  We study properties of this map in \cref{sec61}.

Now for a character $\chi$ of
$\Cl(K)$ and $\phi \in S(\calO_N)$, 
we define a period $P_{K,\chi}(\phi)$ which is
a sum of $h_K$ (not necessarily distinct) values of $\phi$
in terms of this ideal class map.  Assuming $\phi$ is an eigenform,
\cite{wald} tells us that 
$P_{K,\chi}(\phi) \ne 0$ implies $L(\frac 12, \phi_K \otimes \chi) \ne 0$.
Under certain local conditions the converse is also true.

We say an eigenform $\phi \in S(\calO)$ is zero-free if $\phi(x) \ne 0$ for all $x \in \Cl(\calO)$.  From \cref{conj11}, we expect that 100\% of eigenforms $\phi$
with appropriate local sign conditions are zero-free.
It is immediate that if $h_K = 1$ and $\phi$ is zero-free, then $P_{K,1}(\phi) 
\ne 0$.  More generally, orthogonality of characters implies the following.

\begin{thm}[see \cref{thm:existchi}] \label{thm:15}
Suppose $N \in \Sq_\odd$, and
let $K$ be a quadratic field embedding in $B_N$.
Assume $\phi \in S(\calO_N)$ is zero-free.  
Then there exists a character $\chi$ of $\Cl(K)$ such that
$L(\frac 12, \phi_K \otimes \chi) \ne 0$.
\end{thm}

We note that for any $\phi$ with root number $+1$,
the conclusion follows for $K$ of sufficiently large discriminant
via equidistribution of the ideal class maps (see \cite{michel},
\cite{michel-venkatesh}).  
However, our emphasis is that the above is valid for any $K$
embedding in $B_N$, which lets 
us conclude nonvanishing for specific $K$.

The most interesting case is when $\chi=1$, as then one can
conclude both $L(\frac 12, \phi) \ne 0$ and the quadratic Dirichlet twist
$L(\frac 12, \phi \otimes \eta_D) \ne 0$.  Using properties of the
ideal class map, we can show the following.
Let us call an imaginary quadratic $K=\Q(\sqrt{-D})$ special for $B$ if either 
(i) $h_K = 1$ and $K$ embeds in $B$,
or (ii) $K$ has one class per genus (i.e., the class group has exponent 2) 
and $p | D$ implies $p | N$.  (The latter condition implies $K$ embeds in $B$.)
E.g., the class number 2 field $K=\Q(\sqrt{-5})$ is special for any $B_N$ such
that $10 | N$.

\begin{thm}[see \cref{thm:1cpg}] \label{thm:16}
Suppose $N \in \Sq_\odd$ and $\phi \in S(\calO_N)$ is zero-free.  Then 
$L(\frac 12, \phi_K) \ne 0$ for all $K$ which are special for $B_N$.
\end{thm}

While there is no apparent way of expressing the condition that
$\phi \in S(\calO_N)$ is zero-free directly in terms of the associated newform
$f \in S_2(N)$, we briefly
explain two ways one can computationally verify this from the perspective of
classical modular forms.
First, suppose $p \nmid N$ is such that $T_p$ acts on $S_2^\new(N)$ with distinct Hecke
eigenvalues.  One can realize $T_p$ as a Brandt matrix, and $\phi$ as an eigenvector
of this matrix for the eigevalue $a_p(f)$.  So $\phi$ being zero-free means that an 
eigenvector of a Brandt matrix has no zero entries.  This is the direct, computational
approached we used to generate our data.

Here is another situation in which one can check $\phi$ is zero-free using some of our
results below.  Suppose the Atkin--Lehner eigenvalues $w_p(f) = -1$ for all $p | N$,
or that the dimension of the minimal Atkin--Lehner eigenspace containing $f$ is one 
more than the dimension of the Atkin--Lehner eigenspace with all Atkin--Lehner eigenvalues $w_p = -1$.  If the Galois conjugates of
$f$ generate the whole Atkin--Lehner eigenspace, then $\phi$ is zero-free by
\cref{thm:nozeroes} and the condition in \eqref{notriv0cond}.  
In either of these situations, \cref{thm:16} tell us that
$L(\frac 12, f)L(\frac 12, f \otimes \eta_K) \ne 0$ for any special $K$ embedding in $B$.

Note that there are only finitely many $K$ which are special for a given $B$.
In fact, there are at most 66 one-class-per-genus imaginary quadratic fields
\cite{weinberger}, so the collection of $K$ to which \cref{thm:16} applies is rather
small.  However, this still may be of some interest as 
there are very few general results about nonvanishing of
specific twists of specific forms which do not use special arithmetic
constructions such as Eisenstein congruences.  This is somewhat general
in the sense that, for any given one-class-per-genus $K$, there are infinitely 
many $B$ for which $K$ is special.  The analytic rank of $\phi$ is 
defined to be the order of vanishing of $L(s, \phi)$ at $s = \frac 12$.
If $B$ is such that there exists
a special $K$ for $B$ (which holds for more than 99\% of prime discriminant
$B$), then $\phi$ being zero-free is a sufficient condition for having analytic
rank 0 (i.e., $L(\frac 12, \phi) \ne 0$).

Being more speculative, we can view the proportion of values of an eigenform $\phi \in S(\calO_N)$ which are 0 as a proxy for the probability that a special
base change $\phi_K$ has analytic rank $\ge 2$, assuming each local sign
is $+1$ (see \cref{prop:per-van}).  (This proxy is not meant to be taken
too literally, but experimentally rank 2 forms tend to have more zeroes
than rank 0 forms.)
Thus the proportion of values of $\phi$
which are nonzero is a proxy for an upper bound on the probability that
$\phi$ has analytic rank $\ge 2$.  This combined with \cref{conj12}
(or rather the more precise \cref{conj:deg1}), and the proliferance of rank 2 elliptic curves, suggests that most high rank newforms come from elliptic curves:

\begin{conj} \label{last-conj}
Among all weight $2$ newforms with analytic rank $\ge 2$,
$100\%$ have degree 1---i.e., rational Fourier coefficients---when partially ordered by level.
\end{conj}

To elaborate a little more on our reasoning for \cref{last-conj}, from 
\cite{me:maeda} we expect
that for $N$ squarefree the set of newforms in $S_2^\new(N)$ with fixed
Atkin--Lehner signs $w_p$ for all $p | N$ should decompose into 1 large Galois
orbit and at most a few small Galois orbits.  As $N \to \infty$, it is conjectured
that 50\% of newforms in $S_2^\new(N)$ 
have analytic rank $0$ and 50\% have analytic rank $1$ (e.g., see \cite{ILS}).
Consequently, one---and therefore every---form in a large Galois orbit should
have analytic rank $\le 1$, i.e., all newforms with analytic rank $\ge 2$ should 
have small degree.  

Now \cite{me:maeda} suggests most small degree newforms are
in fact degree 1.  Further, \cref{conj:deg1} suggests that the probability,
as a function of $N$, that a degree $d$ form has rank $\ge 2$ tends to 0
notably faster for $d > 1$ than for $d=1$.  These expectations lead to
\cref{last-conj}.

We remark that, according to data in the LMFDB, for weight 2 newforms of level 
$\Gamma_0(N)$ with $N < 10000$ (not necessarily in $\Sq_\odd$ or squarefree), there are 2436 Galois orbits with 
analytic rank at least 2.  (In fact, only one has rank bigger than 2, 
corresponding to the rank 3 elliptic curve of conductor 5077.)  Of these 1970 
are degree 1, 372 are degree 2, 74 are degree 3,
18 are degree 4, and 2 are degree 5.  These data, while not ample,
agree with the above expectations.

\subsection*{Acknowledgements}
We thank David Farmer for useful comments.  We also appreciate the detailed
and thoughtful feedback of the referee.
The first author was supported by a grant from the 
Simons Foundation/SFARI (512927, KM).

\subsection*{Notation}
Throughout $B$ denotes a definite quaternion algebra over $\Q$
and $\calO$ an order in $B$.  Initially $\calO$ is allowed to
be an Eichler or sometimes more general order.
However, from \cref{sec:dims} onwards we assume $\calO$ is 
maximal.  
Denote by $\Cl(\calO) = \{ x_1, \dots, x_h \}$ the set of invertible right $\calO$-ideal
classes in $B$, and $h = |\Cl(\calO)|$ the class number of $\calO$.

Let $\Ram(B)$ be the set of finite rational primes ramified in $B$.
Let $\Delta = \prod_{p \in \Ram(B)} p$  be the discriminant of $B$,
and $N$ the level of $\calO$.  Note that some authors normalize
levels so that the level of a maximal order is 1, and some so that
the level of a maximal order is $\Delta$---we use the latter convention.

For us, $p$ always denotes a finite rational prime.
For a $\Z$-module $M$, we denote by $M_p = M \otimes_{\Z} \Z_p$
the localization at $p$.  Set $\hat \calO^\times = \prod \calO_p^\times$
and $\hat B = \prod' B_p^\times$, where $\prod'$ means 
the restricted direct product with respect to $\{ \calO_p^\times \}$.  

Similarly, if $K$ is a number field, we let $\frako_K$ be its integer ring,
$\Cl(K)$ its class group, $h_K$ its class number, and let
$\hat \frako_K^\times$ and $\hat K^\times$ denote the usual restricted products 
of the localized multiplicative groups over finite primes.
In \cref{sec:Lvals}, $K = \Q(\sqrt{-D})$ will always denote an imaginary quadratic
field of discriminant $-D$.  

For $N \in \Z_{>0}$, denote by $\omega(N)$ the number of prime divisors of $N$.

With apologies to the reader, $\eps$ is used to denote a sign or
a collection thereof, and $\epsilon$ a positive real number.


\section{Quaternionic modular forms}


Here we review the theory of definite quaternionic modular forms of trivial weight
and make some basic observations.  
Some references for quaternionic modular forms and their connection to
classical modular forms are \cite{dembele-voight}, \cite{me:cong}
and \cite{me:basis}.
For simplicity, we work over $\Q$, and primarily with
Eichler orders, but the theory extends to totally real number fields and more general
orders.

Let $B/\Q$ be a definite quaternion algebra of discriminant $\Delta$, and
$\calO$ an order in $B$.
We identify the set $\Cl(\calO)$ of invertible right $\calO$-ideal classes
with $B^\times \bs \hat B^\times / \hat \calO^\times$.
Here  $B^\times$ is diagonally embedded in $\hat B^\times$.
For ease of notation, for a function $f$ on $\Cl(\calO)$, we write
$f(x)$ for $f(B^\times x \hat \calO^\times)$ where $x \in \hat B^\times$.

We define the space of quaternionic modular forms on $B$ of level $\calO$ 
(with trivial weight and trivial central character) to be
\[ M(\calO) = \{ \phi : \Cl(\calO) \to \C \}. \]
Let $x_1, \dots, x_h \in \hat B^\times$ be a set of representatives 
for $\Cl(\calO)$, and fix right $\calO$-ideal class representatives $\calI_i$ corresponding to each $x_i$.  Let $e_i = [\calO_l(\calI_i)^\times : \Z^\times]$, where
$\calO_l(\calI_i)$ denotes the left order of $\calI_i$, i.e.,
$\calO_l(\calI_i) = x_i \hat \calO x_i^{-1} \cap B$.

Consider the Petersson inner product on $M(\calO)$ given by
\[ (\phi, \phi') = \sum_{i=1}^h \frac 1{e_i} \phi(x_i) \overline{\phi'(x_i)}. \]
For $\alpha \in \hat B^\times$, define the Hecke operator
\[ T_\alpha(\phi)(x) = \sum \phi(x \beta), \]
where $\beta$ runs over a set of representatives for 
$\hat \calO^\times \alpha \hat \calO^\times / \hat \calO^\times$ so that
$\hat \calO^\times \alpha \hat \calO^\times = \bigsqcup \beta \hat \calO^\times$.
For a positive integer $n$, define $T_n = \sum T_\alpha$ where $\alpha$ runs over
a set of elements in $\hat B^\times$ such that $\bigsqcup \hat \calO^\times \alpha
\hat \calO^\times = \{ x \in \hat \calO : |N_{B/\Q}(x)| = n \}$.  Here
$N_{B/\Q}$ denotes the reduced norm from $B$ to $\Q$.
  Classically, we may view
$T_n \phi$ as the sum of right translates of $\phi$ by the integral
right $\calO$-ideals of norm $n$.
The Hecke operators $T_n$ are a commuting family of self-adjoint
operators with respect to $( \cdot \,, \cdot )$, which are algebraically
generated by the $T_p$'s.

For simplicity, we now assume $\calO$ is Eichler or more generally
special of unramified quadratic type in the terminology of \cite{me:basis}.  
Let $N$ be the level of $\calO$.  Then the class number $h$ of $\calO$
only depends on $\Delta$ and $N$.  See \cite{HPS} for a formula for $h$.
We call the Hecke algebra generated over $\Z$ by all $T_p$'sacting on $M(\calO)$ or $S(\calO)$ the full Hecke algebra.
The subalgebra generated by all $T_p$'s with $p \nmid N$ is called the unramified
or spherical Hecke algebra. 

Let $\one \in M(\calO)$ denote the constant function 1.  Then
$\C \one$ is a Hecke-invariant subspace of $M(\calO)$
(for $p \nmid N$, a local calculation shows there are $p+1$
right $\calO$-ideals of norm $p$, so $T_p \one = (p+1) \one$), which we think of
as the Eisenstein subspace of $M(\calO)$.
We define the space $S(\calO)$ of cusp forms to be the orthogonal
complement of $\one$ in $M(\calO)$, i.e., $\phi \in S(\calO)$ if and
only if $\sum \frac 1{e_i} \phi(x_i) = 0$.
By an eigenform of $M(\calO)$, we mean an element $\phi$ of $M(\calO)$
which is a simultaneous eigenvector of all $T_p$'s for $p \nmid N$. 
If $\calO$ is maximal, this is equivalent to saying that
$\phi$ is a simultaneous eigenvector for every $T_p$ (or every $T_n$).

The (Eichler--Shimizu--)Jacquet--Langlands correspondence induces
a monomorphism of Hecke modules from $S(\calO) \to S_2(N)$ for the unramified Hecke algebra, and one can describe the image as
a Hecke module as in \cite{me:basis}.  In the special case that $\calO$ is maximal,
we in fact get Hecke module isomorphisms with newspaces for the
full Hecke algebra:
\begin{equation} \label{eq:JL}
M(\calO) \simeq \C E_{2,N} \oplus S_2^\new(N), \quad
S(\calO) \simeq S_2^\new(N),
\end{equation}
Here $E_{2,N}$ is an Eisenstein series in $M_2(N)$ with the
appropriate ramified Hecke eigenvalues (see
\cite{me:cong} or \cite{me:basis}).  For $N$ prime, we may take
$E_{2,N}$ the unique Eisenstein series in $M_2(N)$ with first Fourier coefficient
1.

\begin{example} \label{ex:37}
Let $N=37$, $B$ be the quaternion algebra of discriminant $N$, and 
$\calO$ be a maximal order in $B$.  For prime discriminants $N$,
one can use the Eichler mass formula and class number formula to
determine $h$ and the $e_i$'s in terms of the congruence class of 
$N \mod 12$.  E.g., see \cite[Table 1.3]{gross}.   In this case,
the class number is $h=3$ and each $e_i = 1$.
Hence $S(\calO) = \{ \phi \in M(\calO) : \sum_{i=1}^3 \phi(x_i) = 0 \}$.

We compute (the Brandt matrix for) the
 Hecke operator $T_2$ in Sage with the command  
 \verb+BrandtModule(37).hecke_matrix(2)+.  This gives
 $T_2 = \bmx 1 & 1 & 1 \\ 1 & 0 & 2 \\ 1 & 2 & 0 \emx$.
 Since $T_2$ acts on $\C E_{2,N} \oplus S_2^\new(N)$, and thus
 $M(\calO)$, with distinct eigenvalues, it generates the Hecke algebra over $\C$,
 and $\phi$ in $M(\calO)$ will be an eigenform if and only if
 $(\phi(x_1), \phi(x_2), \phi(x_3))$ is an eigenvector of $T_2$.
Hence a basis of eigenforms for $M(\calO)$ is given by the following table:
 \begin{center}
\begin{tabular}{c|ccc|c}
& $x_1$ & $x_2$ & $x_3$ & $\eps_{37}$  \\
\hline
$\one$ & 1 & 1 & 1  & $+1$  \\
$\phi_1$ & 2 & $-1$ & $-1$ & $+1$  \\
$\phi_2$ & 0 & $1$ & $-1$ & $-1$ \\
\end{tabular}
\end{center}
The $x_i$ column represents the value $\phi(x_i)$, and the $\eps_{37}$ column
represents the eigenvalue for $T_N = T_{37}$, which for $N$ prime gives the
root number of the eigenform.  

Note $S_2(37)$ contains two newforms $f_1, f_2$.
These are both rational (degree 1) and thus correspond to the two isogeny classes of 
elliptic curves of conductor 37.  Say $f_1$ has root number $+1$ and $f_2$ has
root number $-1$.  Then the Jacquet--Langlands correspondence says that,
for all Hecke operators $T_n$,
$\one$ has the same Hecke eigenvalues as $E_{2,37}$, and $\phi_i$ has the
same Hecke eigenvalues as $f_i$ for $i=1, 2$.
\end{example}

The following lemma is valid for arbitrary orders.

\begin{lem} \label{lem1}
If $\calO$ and $\calO'$ are locally isomorphic orders in $B$ (i.e., $\calO_p \simeq \calO'_p$ for all $p$), then there
is a bijection $\iota : \Cl(\calO') \to \Cl(\calO)$ such that
$\phi \mapsto \phi \circ \iota$ defines a Hecke isomorphism
$\iota^*: M(\calO) \to M(\calO')$ for the full Hecke algebra,
i.e., $\iota^*(T_n \phi) = T_n (\iota^* \phi)$ for all $n$.
\end{lem}

\begin{proof} Since $\calO$ and $\calO'$ are locally isomorphic, there exists
$\xi \in \hat B^\times$ such that $\hat \calO' = \xi \hat \calO \xi^{-1}$.
The map $x \mapsto x \xi$ for $x \in \hat B^\times$ induces a bijection
$\iota : \Cl(\calO') \to \Cl(\calO)$.  

If we let $\gamma \in \hat \calO/\hat \calO^\times$ 
run over the integral right $\calO$-ideals of norm $n$,
then
\[ \iota^*(T_n \phi)(x) = (T_n \phi)(x \xi) = \sum \phi(x \xi \gamma) = 
\sum \phi (x \gamma' \xi) = T_n(\iota^*\phi)(x), \]
where $\gamma' = \xi \gamma \xi^{-1}$, as then
$\gamma'$ runs over the integral right $\calO'$-ideals of norm $n$.
\end{proof}

The genus of $\calO$ is the collection of all orders $\calO'$ in $B$
which are locally isomorphic to $\calO$. The
number of isomorphism classes of orders in the genus of $\calO$ is finite,
and is called the type number of $\calO$.

Consequently, with the notation as in the lemma,
if we take a basis of eigenforms $\{ \phi_i \}$ for $M(\calO)$, then
$\{ \phi_i \circ \iota \}$ is a basis of eigenforms for $M(\calO')$.
If $\calO$ is Eichler (or special of unramified quadratic type), then
the genus of $\calO$ simply consists of all Eichler (or special of
unramified quadratic type) orders $\calO' \subset B$ of level $N$
\cite{HPS}.  Now \cref{lem1} tells us that 
the sets of values of the eigenforms in $M(\calO)$ only
depend upon the genus of $\calO$, which for our class of orders 
only depends upon $\Delta$ and $N$.

Suppose now $\calO$ is maximal (so $N=\Delta$).
Then, given any two eigenforms $\phi, \phi' \in M(\calO)$,
their collections of Hecke eigenvalues agree for all $T_n$ 
(or even almost all $T_p$)
if and only if $\phi$ and $\phi'$ are scalar multiples.
Thus the multiset of values $\{ \phi(x_1), \dots, \phi(x_h) \}$
of an eigenform $\phi$ for a maximal order $\calO$ is, up to a scalar multiple,
determined by its Hecke eigenvalues.  In particular, the set of zeroes
of $\phi$ is determined its Hecke eigenvalues.  Moreover,
the number of zeroes of such a $\phi$ is an invariant
of the associated newform $f \in S_2(N)$, i.e., it does not depend upon
the choice of $\calO$ or the choice of $\phi$ (provided $\Delta = N$).

In \cref{sec:Lvals}, we also consider quaternionic modular forms 
from the representation-theoretic perspective. 
Trivial weight automorphic forms for $B^\times(\A_\Q)$ with trivial character
are simply functions $\phi : B^\times \A_\Q^\times \bs B^\times(\A_\Q)/B^\times(\R) 
\simeq B^\times \hat \Q^\times \bs \hat B^\times \to \C$.  
Via the right regular representation, this space of automorphic forms
decomposes into a direct sum of irreducible automorphic representations 
$\pi$, and  correspondingly $M(\calO)$ breaks up as a direct
sum of $\hat \calO^\times$-invariant functions:
$M(\calO) = \bigoplus \pi^{\hat \calO^\times}$.  Similarly, $S(\calO) = \bigoplus \pi^{\hat \calO^\times}$, where $\pi$ runs over the irreducible automorphic representations
of $B^\times(\A_\Q)$ which are not 1 dimensional (so necessarily infinite dimensional).


\section{Local involutions and trivial zeroes}
\label{sec:triv0}


In this section, we study the zeroes of quaternionic modular forms
forced by local sign conditions, which we call trivial zeroes.
Recall $\Delta$ is the discriminant of $B$.
We assume now $\calO \subset B$ is Eichler of level $N$.  From
\cref{sec:dims} onwards, we will further assume $\calO$ is maximal,
i.e., $N=\Delta$.

\subsection{Local involutions and trivial zeroes} 
Consider a prime $p | N$. If $p | \Delta$, then
$\calO_p$ is the unique maximal order in $B_p$, and we let $\varpi_p = \varpi_{B_p}$ 
denote a uniformizer in $\calO_p$.  If $p \nmid \Delta$, then
$\calO_p \subset B_p = M_2(\Q_p)$ is a local Eichler order of some level $p^e$, 
and we can choose
$\varpi_p = \varpi_{\calO_p} \in \calO_p$ such that $\varpi_p$ is 
$\GL_2(\Q_p)$-conjugate to $\bmx & 1 \\ p^e \emx$
and $\varpi_p$ normalizes $\calO_p$. 

In either case, we lift our local $\varpi_p \in \calO_p$ to the element
$\hat \varpi_{p} = (1, \dots, 1, \varpi_{p}, 1, \dots) \in \hat B^\times$,
which normalizes $\hat \calO^\times$.
We may assume $\varpi_{p}^2 = p^{v_p(N)} \in Z(B_p^\times) \simeq \Q_p^\times$.
Thus right multiplication by $\hat \varpi_{B_p}$  gives an action on $\Cl(\calO)$ 
which is a permutation of order at most 2.
Denote this involution by $\sigma_p$, which we will also view as the
corresponding permutation of our fixed set of representatives
$\{x_1, \dots, x_h \}$.

We can also view this action classically in terms of ideals.  Let $\mathrm{Pic}(\calO)$
denote the Picard group of 2-sided invertible $\calO$-ideals modulo $\Q^\times$.
Each $\hat \varpi_{B_p}$ corresponds to a 2-sided (prime) $\calO$-ideal $\mathfrak P_p$ 
of norm $p$,
and these ideals form a set of generators for $\mathrm{Pic}(\calO) \simeq (\Z/2\Z)^{\omega(N)}$.  
Then $\mathrm{Pic}(\calO)$ acts on $\Cl(\calO)$
by right multiplication, with multiplication by the class of the $\mathfrak P_p$ acting
by $\sigma_p$.

If $p | \Delta$, the Hecke operator $T_{\hat \varpi_p} = T_p$, with $T_p$ defined
as above.  
However, if $p | N$ but $p \nmid \Delta$, then $T_{\hat \varpi_p}$ is not the $T_p$
defined in the previous section, but is the analogue of the Atkin--Lehner operator 
$W_p$ on $M(\calO)$.  E.g., if $p || N$ but $p \nmid \Delta$,
then $T_{\hat \varpi_p}$ acts as $-T_p$ on $S(\calO)$ and as the identity on $\C \one$.

Now $\phi \in M(\calO)$ being an eigenform for $T_{\hat \varpi_p}$ for some $p | N$
with eigenvalue $\eps_p = \pm 1$ is equivalent
to the statement that
\begin{equation} \label{Tpram}
 \phi(\sigma_p(x_i)) = \eps_p \phi(x_i), \quad 1 \le i \le h. 
\end{equation}
In particular, if $x_i$ is a fixed point of the involution $\sigma_p$, then
any $\phi$ in the $(-1)$-eigenspace for $T_{\hat \varpi_p}$ must necessarily vanish on $x_i$.  For instance, in \cref{ex:37}, one can deduce from \eqref{Tpram}
that $\sigma_{37}$ permutes $x_2$ and $x_3$ and fixes $x_1$, which forces
$\phi_2(x_1) = 0$.
This is the simplest way in which the local involutions $\sigma_p$
can force an eigenform $\phi$ to have zeroes, but there are others.

For instance, suppose $p$ and $q$ are primes dividing $\Delta$,
 and $\sigma_p$ and $\sigma_q$
both interchange $x_1$ and $x_2$.  If the $T_p$ and $T_q$
eigenvalues of $\phi$ have opposite signs, then we are forced
to have $\phi(x_1) = \phi(x_2) = 0$ (as in \cref{ex:154} below).
Alternatively, suppose $p$, $q$ and $r$ are primes dividing $\Delta$, and 
$\sigma_p$ (resp.\ $\sigma_q$, resp.\ $\sigma_r$) interchanges $x_1$ and $x_2$
(resp.\ $x_2$ and $x_3$, resp.\ $x_3$ and $x_1$).
Then if $\phi$ is a form in the $(-1)$-eigenspace for $T_p$ and the $(+1)$-eigenspaces 
for $T_q$ and $T_r$, we must have
\[ \phi(x_1) = - \phi(x_2) = - \phi(x_3) = - \phi(x_1), \]
forcing $\phi(x_1)$ (and thus $\phi(x_2)=\phi(x_3)$) to be zero.  

Let $X_1, \dots, X_t$ denote the set of orbits of $\Cl(\calO)$
under the action of $\mathrm{Pic}(\calO)$.  We remark that
two ideal classes in $\Cl(\calO)$ have isomorphic left orders
if and only if they lie in the same orbit under $\mathrm{Pic}(\calO)$.
Consequently, $t$ is the type number of $\calO$.

By a sign pattern $\eps$ for $N$, we mean
a multiplicative function $d \mapsto \eps_d$ from the set of divisors $d$ of $N$ to 
$\{ \pm 1 \}$.
Consider the associated eigenspace for $(T_{\hat \varpi_p})_{p | N}$,
\[ M^\eps(\calO) = \{ \phi \in M(\calO) : T_{\hat \varpi_p}(\phi) = \eps(p) \phi \text{ for } p | N \}. \]

Now, as in \cite{me:cong2}, we define a signed (multi)graph $\Sigma^\eps$ on $\Cl(\calO)$
as follows.  Begin with an empty graph on $\Cl(\calO)$.  
For $p | N$, we adjoin an edge between $x_i$ and $\sigma_p(x_i)$ with
sign $\eps_p $ for each orbit $\{ x_i, \sigma_p(x_i) \}$ of $\sigma_p$.   
(If $x_i$ is a fixed point of $\sigma_p$, we are adjoining a signed loop.)
Call the resulting graph $\Sigma^\eps$.  Note that its connected components are
$X_1, \dots, X_t$.

For a (not necessarily simple) path in $\Sigma^\eps$, define its parity to 
be the product of the signs of its edges.  By a cycle we will simply mean a closed path,
(allowing for loops as well as repeated vertices and edges). 
 We call a component $X_j$ $\eps$-admissible
if there are no cycles of parity $-1$ in $\Sigma^\eps$ that are supported on $X_j$.

Consider any $\phi \in M^\eps(\calO)$.  If $x, x' \in X_j$, then by definition there
is a sequence $p_1, \dots, p_m$ of (not necessarily distinct) primes dividing $N$
such that $(\sigma_{p_m} \circ \dots \circ \sigma_{p_1})(x) = x'$.
Consequently, the values of $\phi$ on all ideal classes in a given $X_j$ 
are determined by the value on a single $x \in X_j$, and in fact
$\phi(x') = \pm \phi(x)$ for $x, x' \in X_j$.  Here the sign must be the parity
of the path in $\Sigma^\eps$ given by
$x, \sigma_{p_1}(x), \dots, (\sigma_{p_m} \circ \dots \circ \sigma_{p_1})(x)$.  
In particular, when $x'=x$ this path is a cycle, and if it has parity $-1$, then
we must have $\phi(x) = 0$.  We call such a zero of $\phi$ a trivial zero,
i.e., we say $\phi(x)$ is a trivial zero if $\phi \in M^\eps(\calO)$ and
the $\mathrm{Pic}(\calO)$-orbit of $x$ is not $\eps$-admissible.
Since this notion of a trivial zero only depends on $\eps$ and the orbit
$X$ rather than on $\phi$ and $x$, 
we sometimes simply say that $x$ or $X$ is a trivial zero for $\eps$.

\begin{lem} There exists some $\phi \in M^\eps(\calO)$ which is nonzero on
$X_j$ if and only if $X_j$ is $\eps$-admissible.
\end{lem}

\begin{proof} The ``only if'' direction was shown above.  
Conversely, suppose $X_j$ is $\eps$-admissible.  Define a nonzero $\phi$ supported
on $X_j$ as follows.  Fix some $x \in X_j$ and set $\phi(x) = 1$.  
For any $x' \in X_j$, choose
a sequence of $p_1, \dots, p_m$ of primes dividing $N$ which corresponds to
a path from $x$ to $x'$ as above.  Put $\phi(x') = \pm 1$ where $\pm$ is the
parity of this path in $\Sigma^\eps$.  Note that, by admissibility, the parity is independent 
of the choice of path (otherwise we could concatenate them to get a cycle of parity $-1$).
In particular, $\phi(\sigma_p(x_i)) = \eps_p \phi(x_i)$ for $p | N$ and $1 \le i \le h$,
i.e., $\phi \in M^\eps(\calO)$.
\end{proof}

Since we may take a basis for $M^\eps(\calO)$ which consists of forms $\phi$
that are supported on a single connected component $X_j$, this implies
that $\dim M^\eps(\calO)$ equals the number of $X_j$ which are $\eps$-admissible.  Thus the set of trivial zeroes for $\eps$, i.e., the union of the 
$\eps$-inadmissible orbits, is precisely the set of $x \in \Cl(\calO)$
such that $\phi(x) = 0$ for all $\phi \in M^\eps(\calO)$.

\begin{example} \label{ex:154}
 Here is a basis of eigenforms for a maximal
order $\calO$ in the quaternion algebra $B$ of discriminant
$\Delta = N=154 = 2 \cdot 7 \cdot 11$, for a suitable ordering of
ideal classes.  The $\eps_p$ columns
denote the associated eigenvalue for $T_p = T_{\hat \varpi_p}$, where $p | N$.
 \begin{center}
\begin{tabular}{c|cccccc|c|c|c}
& $x_1$ & $x_2$ & $x_3$ & $x_4$ & $x_5$ & $x_6$ & $\eps_2$ & $\eps_7$ & $\eps_{11}$ \\
\hline
$\one$ & 1 & 1 & 1 & 1 & 1 & 1 & $+$ & $+$ & $+$ \\
$\phi_1$ & 1 & 1 & $\alpha$ & $\alpha$ & $-2\alpha-2$ &
$-2\alpha-2$ & $+$ & $+$ & $+$ \\
$\phi_2$ & 1 & 1 & $\bar \alpha$ & $\bar \alpha$ & $-2\bar \alpha-2$ &
$-2\bar \alpha-2$ & $+$ & $+$ & $+$ \\
$\phi_3$ & 0 & 0 & 0 & 0 & 1 & $-1$ & $+$ & $-$ & $-$ \\
$\phi_4$ & 1 & $-1$ & 0 & 0 & 0 & 0 & $-$ & $-$ & $+$ \\
$\phi_5$ &0 & 0 & 1 & $-1$ & 0 & 0 & $-$ & $-$ & $-$
\end{tabular}
\end{center}
Here $\alpha = \frac{-3+\sqrt 5}2$ and $\bar \alpha = \frac{-3-\sqrt 5}2$.
One can deduce from this table that there are 3 orbits: 
$X_1 = \{ x_1, x_2 \}$,
$X_2 = \{ x_3, x_4 \}$ and $X_3 = \{ x_5, x_6 \}$.  
Moreover, in cycle notation $\sigma_2$ acts as $(x_1 \, x_2)(x_3 \, x_4)$,
$\sigma_7$ acts as $(x_1 \, x_2)(x_3 \, x_4)(x_5 \, x_6)$,
and $\sigma_{11}$ acts as $(x_3 \, x_4)(x_5 \, x_6)$.  
The signed graph $\Sigma^\eps$ is
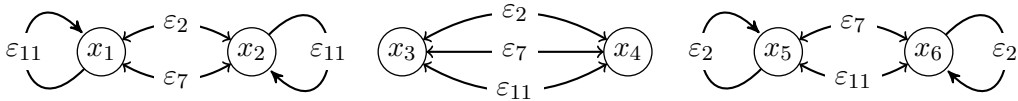
\begin{center}
\begin{tikzpicture}
\useasboundingbox (0,-0.75) rectangle (11,0.75);
\Vertex[L=\hbox{$x_1$},x=0,y=0]{v1}
\Vertex[L=\hbox{$x_2$},x=2,y=0]{v2}
\Vertex[L=\hbox{$x_3$},x=4,y=0]{v3}
\Vertex[L=\hbox{$x_4$},x=7,y=0]{v4}
\Vertex[L=\hbox{$x_5$},x=9,y=0]{v5}
\Vertex[L=\hbox{$x_6$},x=11,y=0]{v6}
\tikzset{LabelStyle/.style =
   {fill=white}}
\Loop[dist=1.5cm,label=$\eps_{11}$](v1)
\Loop[dist=1.5cm,dir=EA,label=$\eps_{11}$](v2)
\Edge[style={bend left, <->},label=$\eps_2$](v1)(v2)
\Edge[style={bend left, <->},label=$\eps_7$](v2)(v1)
\Edge[style={bend left, <->},label=$\eps_2$](v3)(v4)
\Edge[style={bend left, <->},label=$\eps_{11}$](v4)(v3)
\Edge[style=<->,label=$\eps_7$](v3)(v4)
\Loop[dist=1.5cm,label=$\eps_2$](v5)
\Loop[dist=1.5cm,dir=EA,label=$\eps_2$](v6)
\Edge[style={bend left, <->},label=$\eps_7$](v5)(v6)
\Edge[style={bend left, <->},label=$\eps_{11}$](v6)(v5)
\GraphInit[vstyle=empty]
\end{tikzpicture}
\end{center}
In particular, unless all the signs are the same, we
have trivial zeroes on $X_2$---indeed, $\phi_3$ and $\phi_4$
are zero on $X_2$.  If $\eps_{11} = -1$ or if $\eps_2$ and $\eps_7$ have
opposite signs (as for $\phi_3$ and $\phi_5$), we have trivial zeroes on $X_1$. 
Similarly, if $\eps_2 = -1$ or if $\eps_7$ and $\eps_{11}$ have opposite signs (as for $\phi_4$ and $\phi_5$), we have trivial zeroes on $X_3$.
Hence all zeroes in the table are trivial zeroes.
\end{example}

Let $+_N$ (resp.\ $-_N$) denote the sign pattern for $N$ which is $+1$ (resp.\ $-1$) for all divisors (resp.\ prime divisors) of $N$.

\begin{prop} \label{prop33}
Let $\eps$ be a sign pattern for $N$.
Then the number of orbits $X_j$ on which
$\eps$ has trivial zeroes is $\dim M^{+_N}(\calO) - \dim M^\eps(\calO)$.
In particular, $\eps$ has no trivial zeroes if and only if
$\dim M^\eps(\calO) = \dim M^{+_N}(\calO)$.
\end{prop}

In particular, no eigenform in $M^{+_N}(\calO)$ has trivial zeroes.

\begin{proof} As explained above, $\dim M^\eps(\calO)$ is the number of
$\eps$-admissible orbits $X_j$.  Hence it suffices to show that
all orbits are $+_N$-admissible.  However, this is obvious as all cycles have parity
$+1$ for $\Sigma_{+_N}$.
\end{proof}

For example, if $N=\Delta=30$ so $\calO$ is maximal, then $h=2$.  
Let $\eps$ be the sign pattern with $\eps_2 = \eps_5 = -1$ and 
$\eps_3 = 1$.  Then  $\dim M^\eps(\calO) = \dim M^{+_{30}}(\calO) = 1$,
so the eigenform $\phi \in M^\eps(\calO)$ (which corresponds to the
unique newform in $S^\new_2(30)$) has no trivial zeroes.  In fact,
we may scale $\phi$ so that $\phi(x_1) = -\phi(x_2) = 1$, and 
$\phi$ has no zeroes.

\subsection{Dimension formulas} \label{sec:dims}

From now on, we assume $\calO$ is a maximal order in $B$, i.e.,
$N = \Delta$.
In this situation, we use refined dimension formulas from
\cite{me:dim} to get more explicit 
results on trivial zeroes for sign patterns.
The restriction to $\calO$ maximal is largely for simplicity---one
should be able to similarly treat Eichler orders of squarefree level with
an extension of \cite{me:dim}.
In principle, one could also consider Eichler
orders of non-squarefree level, but then the relevant dimension formulas
are much more complicated.

\begin{lem} \label{sigmaN-fixpt}
The involution $\sigma_N = \prod_{p|N} \sigma_p$
 acts on $\Cl(\calO)$ with fixed points.
\end{lem}

\begin{proof} 
By identifying $M(\calO)$ with $\C^h$ via 
$\phi \mapsto (\phi(x_1), \dots, \phi(x_h))$, we see from \eqref{Tpram} that we can
represent $T_N$ by a matrix which agrees with the matrix representation
for the permutation $\sigma_N$ of $\{ x_1, \dots, x_h \}$.  
In particular, $\tr T_N$ is the number of fixed points of $\sigma_N$.

Now
\[ \tr T_N = 1 + \dim S_2^\new(N)^+ - \dim S_2^\new(N)^-, \]
where $S_2^\new(N)^\pm$ is denotes the subspace of forms
with root number $\pm 1$.   However, \cite[Corollary 2.3]{me:dim}
tells us that $\dim S_2^\new(N)^+ \ge  \dim S_2^\new(N)^-$,
so $\tr T_N \ge 1$.  Hence $\sigma_N$ has fixed points.
\end{proof}

Note that this lemma is really special to $\sigma_N$---it is not valid for
a single $\sigma_p$ in general.  E.g., $\sigma_7$ from 
\cref{ex:154} acts with no fixed points.  More generally, see 
\cref{lem:fixed-pt}.

\begin{cor}  \label{cor35}
If $\eps$ is a sign pattern for $N$ such that $\eps_N = -1$,
then any eigenform in $M^\eps(\calO)$ has trivial zeroes.
\end{cor}

In other words, this says that any $\phi$ which corresponds to a newform
$f \in S_2(N)$ with root number $-1$ has trivial zeroes.

\begin{proof}
If $\eps_N = -1$, then any fixed point of $\sigma_N$ must be a trivial zero
for $\eps$.
\end{proof}

In the simple case that $N$ is prime and
$\calO$ is maximal, there are only two sign patterns, $+_N$ and $-_N$,
and we know that all eigenforms for $+_N$ have no trivial zeroes and
all eigenforms in $-_N$ have trivial zeroes.
When $N$ is not prime, the situation is more complicated, as we
explain presently.  

Denote by $a_n(f)$ the $n$-th Fourier coefficient of
an elliptic modular form $f \in M_2(N)$.
Let $E_{2,N}$ denote the normalized weight 2 Hecke eigenform
Eisenstein series in $M_2(N)$ with Fourier coefficients $a_p(E_{2,N}) = 1$ for each
$p | N$.  Put
\[ M_2^\new(N)^* = \C E_{2,N} \oplus S_2^\new(N). \]
By a newform in $M_2^\new(N)^*$, we will mean either a newform in $S_2^\new(N)$
or $E_{2,N}$. 
For $p | N$, let $W_p$ denote the Atkin--Lehner operator on $M_2^\new(N)^*$,
which is defined as usual on $S_2^\new(N)$ and and acts by $-1$ on $\C E_{2,N}$.  
Then $W_p(f) = -a_p(f)$ for each newform in $M_2^\new(N)^*$.

Let $\eps$ be a sign pattern for $N$.  Put
\[ M_2^{\new,\eps}(N)^* = \{ f \in M_2^\new(N)^* : W_p f = \eps(p) f \text{ for } p | N \}, \]
and similarly $S_2^{\new,\eps}(N) = M_2^{\new,\eps}(N)^* \cap S_2^\new(N)$.
Note $\dim M_2^{\new,\eps}(N)^*$ is simply $1+ \dim S_2^{\new,\eps}(N)$ if
$\eps = -_N$ and $\dim S_2^{\new,\eps}(N)$ otherwise.

Denote by $-\eps$  the sign pattern for $N$
which has the opposite signs as $\eps$ at all primes $p | N$.
Then the Jacquet--Langlands correspondence gives an isomorphism of full
Hecke modules
\[ M^\eps(\calO) \simeq M_2^{\new,-\eps}(N)^*. \]
Hence by \cref{prop33}, a form in $M^\eps(\calO)$ will have no trivial zeroes if and only if
$\eps = +_N$ or 
\begin{equation} \label{notriv0cond}
 \dim S_2^{\new,-\eps}(N) = 1 + \dim S_2^{\new, -_N}(N).
\end{equation}
The instances where \eqref{notriv0cond} holds with $\eps \ne +_N$ and $N < 100$ 
occur when  $N \in \{ 30, 42, 70, 78 \}$.  In these examples $\dim S_2^{\new}(N)=1$ but
$\dim S_2^{\new, -_N}(N) = 0$.  The unique newforms of these levels have root number
$+1$ with 1 Atkin--Lehner sign $-1$ and 2 Atkin--Lehner signs $+1$.  
The values of the associated quaternionic modular
forms are given in \cite[Appendix A]{jordan-thesis}, which tabulates quaternionic
eigenforms for maximal orders of levels $< 100$.
These forms have no zeroes, trivial or not.
Here is the next non-trivial instance when \eqref{notriv0cond} holds.

\begin{example} Let $N=105$.  Here $\dim S_2^\new(N) = 3$.  
We order the newforms $f_1, f_2, f_3 \in S_2(N)$ so that $f_1$ has rationality field
$\Q$ and $f_2, f_3$ are the Galois conjugate forms with rationality field $\Q(\sqrt 5)$.
Then $f_1$ has all 3 Atkin--Lehner signs $-1$, and $f_2, f_3$ both have Atkin--Lehner
signs $w_3 = w_5 = +1$ and $w_7 = -1$.

Taking $\calO$ to be a maximal order in the quaternion algebra $B$ of discriminant
$N$, we compute (in Magma or using class number and mass formulas) that
$h=4$ and each $e_i = 1$.  We may order $\Cl(\calO)$
so that a basis of eigenforms is given by the following table:
\begin{center}
\begin{tabular}{c|cccc|c|c|c}
& $x_1$ & $x_2$ & $x_3$ & $x_4$ & $\eps_3$ & $\eps_5$ & $\eps_{7}$ \\
\hline
$\one$ & 1 & 1 & 1 & 1 & $+$ & $+$ & $+$ \\
$\phi_1$ & 1 & 1 & $-1$ & $-1$ & $+$ & $+$ & $+$ \\
$\phi_2$ & 1 & $-1$ & $2-\sqrt 5$ & $-2 + \sqrt 5$ & $-$ & $-$ & $+$ \\
$\phi_3$ & 1 & $-1$ & $2+\sqrt 5$ & $-2 - \sqrt 5$ & $-$ & $-$ & $+$ \\
\end{tabular}
\end{center}
If we assume $f_2, f_3$ are ordered so that $a_2(f_2) = \sqrt 5$ and $a_2(f_3) = 
-\sqrt 5$, then $\phi_i$ corresponds to $f_i$ for each $1 \le i \le 3$.  Here we see that
\eqref{notriv0cond} holds with $\eps$ given by $(\eps_3, \eps_5, \eps_7) = (-, -, +)$,
and the eigenforms in $M^\eps(\calO)$ have no zeroes, trivial or not.
\end{example}

Denote by $\Delta_d$ the discriminant of $\Q(\sqrt{-d})$ and $h(\Delta_d)$ its
class number.

\begin{prop} \label{prop36}
Assume $N$ is odd, and
let $\eps \ne +_N$ be a sign pattern for $N$.  Let
$S$ be the set of divisors $d > 1$ of $N$ 
such that ${\Delta_d \leg p} = -1$ for all $p | \frac Nd$.
Then $\eps$ has no trivial zeroes if and only
if both of the following conditions hold:

(i) $\eps_d = 1$ for all $d \in S$; and

(ii) if $3 | N$, then $\eps_3 = 1$ or $N$ is divisible by some prime
$p \equiv 1 \mod 3$.
\end{prop}

\begin{proof} It follows from Propositions 1.4 and 3.2 of \cite{me:dim}
(or more directly from Theorem 3.3 of \emph{op.\ cit.} if $3 \nmid N$) 
that $2^{\omega(N)}(\dim M^{+_N}(\calO) - \dim M^{\eps}(\calO))$ 
equals
\begin{equation} \label{eq:dim-bias}
\frac 12 \sum_{1 < d | N} (1 - \eps_d) h'(\Delta_d)
b(d, N/d) \prod_{p | \frac Nd} \left( 1 - {\Delta_d \leg p} \right) 
+ \delta_{3|N} \frac{(1 - \eps_3)}{3} \prod_{p | \frac N3} 
\left( 1 - {-3 \leg p} \right). 
\end{equation}
Here we use the same notation as in  \emph{op.\ cit.}:
each $h'(\Delta_d)$ is a weighted class number and $b(d,N/d)$ is
1, 2 or 4 depending only on $d \mod 8$.  For $N$ odd,
$b(d,N/d) = b(d,1)$, where $b(d,1)$ is given by \eqref{eq:bN} below.
In addition $\delta_{3|N}$ means 1 if $3|N$ and 0 otherwise.
The point is that each term in the above sum is $\ge 0$, so
$\dim M^{+_N}(\calO) = \dim M^{\eps}(\calO)$ if and only if
each term is 0.  
Condition (i) is equivalent to each term in the
first sum being 0, and condition (ii) is equivalent to the final term being 0.
\end{proof}

\begin{rem} (a) Since these conditions are complicated, it is not
obvious how frequently they are satisfied.  In Sage, we computed these
conditions for the 820 odd squarefree levels $N < 10000$ which 
are products of 3 primes.\footnote{Code may be found at:
\url{https://math.ou.edu/~kmartin/data/}}
Of these levels,
465 have at least one sign pattern $\eps \ne +_N$ with no trivial
zeroes.  The total number of such sign patterns 559.  Thus it appears
relatively common that for (non-prime) $N$ as above there is at
least one sign pattern besides $+_N$ with no trivial zeroes.  

(b) While we still have similar dimension formulas
when $N$ is even, the issue with the above argument is that the 
constants $b(d,N/d)$ are $-1$, $-2$ or 0 if $N/d$ is even, so one cannot
use a positivity argument.  We note that of the 980 even levels $N \in \Sq_3$
with $N < 10000$, 372 have a sign pattern $\eps \ne +_N$ with no trivial zeroes.
\end{rem}
 
There is a simpler sufficient condition (in addition to the
root number condition) which allows us to say certain sign patterns
must have trivial zeroes.  We recall the following.

\begin{lem} (\cite[Lemma 4.3]{me:cong2}) \label{lem:fixed-pt}
Let $p | N$.  Then $\sigma_p$
acts on $\Cl(\calO)$ without fixed points if and only if

\begin{enumerate}
\item $p$ is odd and ${-p \leg q} = 1$ for some odd prime $q | N$;

\item $p \equiv 7 \mod 8$ and $N$ is even; or

\item $p=2$, ${-2 \leg q} = 1$ for some prime $q | N$, and $N$ has
a prime factor which is $1 \mod 4$.
\end{enumerate}
\end{lem}

Thus if some $p | N$ does not satisfy any of these conditions, i.e., if $\sigma_p$ acts with
fixed points, then any sign pattern $\eps$ for $N$ with $\eps_p =-1$ 
has trivial zeroes (and the number of trivial zeroes is at least the
number of fixed points of $\sigma_p$).

\subsection{Counting trivial zeroes} \label{sec33}
Recall that $\calO$ is maximal.

\begin{prop} \label{triv0-bound}

(i) Fix $r$ odd.  
As $N \to \infty$ among elements of $\Sq_r$, the maximum number of 
trivial zeroes for a sign pattern for $N$ is $\ll_r N^{\frac 12} \log N$.

(ii) Fix $\epsilon > 0$.  As $N \to \infty$ along $\Sq_\odd$, the
maximum number of trivial zeroes for a sign pattern for $N$ is
$O(N^{\frac 12 + \epsilon})$.
\end{prop}

For comparison, we note that the total number $h$ of values for a
quaternionic modular form in $M(\calO)$ is asymptotic to
$\frac{\phi(N)}{12}$, where here $\phi$ is the Euler totient function.

\begin{proof}
Let $N \in \Sq_r$ and $\eps$ be a sign pattern for $N$.
Recall the number of $\mathrm{Pic}(\calO)$-orbits $X_j$ on which $\eps$ has trivial zeroes is
precisely $\dim M^{+_N}(\calO) - \dim M^\eps(\calO)$.
Moreover, since each orbit is generated by $r$ commuting involutions,
each orbit has size at most $2^r$.  Hence for (i) it suffices to show
$\dim M^{+_N}(\calO) - \dim M^\eps(\calO) = O(N^{1/2} \log N)$.
Up to the addition of a bounded term when $N$ is even, this difference
of dimensions is still given by a formula of the form \eqref{eq:dim-bias}
(cf.\ \cite{me:dim}), and thus for a suitable constant $C$ is bounded by
\[ 2^r C \sum_{d | N} h(\Delta_d) \ll 4^r N^{\frac 12} \log N. \]
Here we used the standard upper bound $h(\Delta_d) \ll \sqrt d \log d \le \sqrt N \log N$
and the fact that $N$ has $2^r$ divisors.  This gives (i).  

To get (ii), it suffices to show that $4^{r} = O(N^\epsilon)$ where $r=\omega(N)$
and $N$ is squarefree.  Clearly $N \ge p_1 \dots p_r$, where $p_i$ is the $i$-th
prime number.  The growth of the primorial function $\prod_{i \le r} p_i$ is known
to be $e^{(1+o(1)) r \log r}$, thus $N$ grows at least as fast as $r^r$. 
From this it follows that $4^r = O(N^\epsilon)$ for any $\epsilon > 0$.
\end{proof}

\begin{rem} In fact one can say a bit more about the sizes of orbits of $\mathrm{Pic}(\calO)$
and trivial zeroes.  First each orbit has size $2^j$ for some $j \le r$ (see 
\cite[Section 4.4]{me:cong2}).  One can show that an orbit of size $2^j$
is inadmissible for some sign pattern $\eps$ if and only if $j < r$.
So one can replace $2^r$ by $2^{r-1}$ in the bound in the above proof.
Further, from the asymptotic equidistribution of sign patterns (see \emph{op.\ cit.}) 
one sees
that almost all orbits must have maximal size $2^r$ as 
$N \to \infty$ along $\Sq_r$.
\end{rem}

Let us now treat the case where $N$ is prime in more detail.  
In this case, there is only a single local involution
$\sigma_N$ to consider, and we know from 
\cref{sigmaN-fixpt} (or \cref{lem:fixed-pt}) that $\sigma_N$
necessarily has fixed points.  Moreover the sign pattern $+_N$ has
no trivial zeroes, and the number of trivial zeroes of $-_N$ must
be the number of fixed points of $\sigma_N$.

For $d$ odd, define
\begin{equation} \label{eq:bN}
 b(d,1) = \begin{cases}
1 & d \equiv 1 \mod 4 \\
2 & d \equiv 7 \mod 8 \\
4 & d \equiv 3 \mod 8.
\end{cases}
\end{equation}

\begin{prop} \label{prop:nfixedpoints}
For $N > 3$ prime, the sign pattern $-_N$ has
exactly $\frac 12 h(\Delta_N) b(N,1)$ trivial zeroes.
\end{prop}

\begin{proof}
The number of fixed points of $\sigma_N$ is
$\dim M^{+_N}(\calO) - \dim M^{-_N}(\calO)$.
This is precisely $\frac 12 h(\Delta_N) b(N,1)$,
which is a special case of either \eqref{eq:dim-bias}
or \cite[Theorem 2.2]{me:dim}.
\end{proof}

Note that this exact formula in the prime level case tells us that
\cref{triv0-bound} is close to being sharp.  

\begin{cor} \label{triv0-prime-lb}
For any $\epsilon > 0$, as $N \to \infty$ along primes,
the number of trivial zeroes for the sign pattern $-_N$ is
$\gg N^{\frac 12 - \epsilon}$.
\end{cor}


\section{Bounds on nontrivial zeroes}


As before, $\calO \subset B$ is a maximal order of level $N$.
We first note the following trivial upper bounds 
on the number of zeroes of an arbitrary (not necessarily eigen) form $\phi$.

\begin{lem} 
Let $\phi \in S(\calO)$ be nonzero.  Then $\phi$ is nonzero on at least
$\max\{2, \frac{(\phi, \phi)}{\norm{\phi}^2} \}$ elements of $\Cl(\calO)$.
\end{lem}

\begin{proof} 
Let $n_\phi$ denote the number of zeroes of $\phi$.
The lower bound of $\frac{(\phi, \phi)}{\norm{\phi}^2}$ follows
from $(\phi, \phi) \le \sum |\phi(x_i)|^2 \le (h-n_\phi){\norm{\phi}^2}$ and
the lower bound of $2$ follows from 
from $(\phi, \one) = \sum \frac 1{e_i} \phi(x_i) = 0$.
\end{proof}

We also note that if $\phi \in S(\calO)$ is a nonzero eigenform, we may
scale its entries to be real.  Then at least one value must be strictly positive
and at least one value must be strictly negative.

By sup norm bounds \cite{blomer-michel}, for an eigenform $\phi$ 
we have $\frac{(\phi, \phi)}{\norm{\phi}^2} \gg h^{\delta}$ for
any $\delta < \frac 1{24}$.
This gives nontrivial bounds for the number of zeroes of an eigenform,
but in fact we expect something much stronger is true.

\subsection{Galois orbits of eigenforms}
Let $\phi \in M(\calO)$ be an eigenform, and $f \in S_2(N)$ the corresponding
newform.  For any $n \ge 1$, let $\lambda_n$
be the $T_n$-eigenvalue for $\phi$.   Let $K_\phi$ be the field generated
by all Hecke eigenvalues $\lambda_n$.  By \eqref{eq:JL}, we know each
$\lambda_n = a_n(f)$, and thus $K_\phi = K_f$ is also the rationality field of $f$.

We normalize $\phi$ so that all of its values lie in $K_\phi$.
To see that this is possible, 
let $V$ be the space of all $K_\phi$-valued functions in $M(\calO)$.  Note that
each $T_{n}$ acts on $V$, since each $T_n$ is integral (i.e., every value of $T_n\phi$
is an $\Z$-linear combination of values of $\phi$).  
Let $V_0 \subset V$ be the intersection of
all of the $\lambda_{n}$-eigenspaces for the $T_{n}$'s.  Note that
$\dim_{K_\phi} V_0 = \dim_{\C} V_0 \otimes \C$.  The latter space is 1-dimensional
by \eqref{eq:JL} and is spanned by $\phi$, so some nonzero scalar multiple of $\phi$ lies in 
$V_0$, i.e., takes values in $K_\phi$.  While not necessary at present, we may also 
scale $\phi$ so that its values lie in the integer ring $\frako_{K_\phi}$.  
We also note that $K_\phi$ is the minimal number field which can contain
all values of a nonzero scalar multiple of $\phi$, as every $\lambda_n$ must lie must lie
in the field generated by the values of $\phi$ by the integrality of the Hecke operators.

Given an element $\tau \in \Gal(\C/\Q)$,
the Galois conjugate form $\phi^\tau(x) := \tau(\phi(x))$ 
is an eigenvector for each $T_n$ with eigenvalue $\tau(\lambda_n)$.
Hence $\Gal(\C/\Q)$ acts on (lines of) eigenforms, and partitions the set of 
(lines of) eigenforms into Galois orbits.
Writing $K_\phi = \Q(\alpha)$ for some algebraic number $\alpha$, we see the
number of distinct Galois conjugates $\phi^\tau$ equals the number of embeddings
of $\alpha$ into $\C$, i.e., 
the degree $[K_\phi : \Q]$ of the minimal polynomial of 
$\alpha$.  In other words, the Galois orbit $\{ \phi^\tau : \tau \in \Gal(\C/\Q) \}$ of 
$\phi$ has size
$\deg(\phi) := [K_\phi : \Q]$, which is the (rationality) degree of $\phi$ (or of $f$).

Clearly the zeroes of an eigenform $\phi$ are the same as the zeroes of any 
conjugate $\phi^\tau$.   Also note that, since the full Hecke algebra
is commutative, if $\phi \in M^\eps(\calO)$, then each conjugate $\phi^\tau
\in M^\eps(\calO)$.

\subsection{Fundamental domains of eigenforms}
For a sign pattern $\eps$ for $N$, let $\calF_\eps$ be a set of representatives for the $\eps$-admissible orbits of $\mathrm{Pic}(\calO)$ acting on $\Cl(\calO)$.
Thus $\# \calF_\eps = \dim M^\eps(\calO)$, and any $\phi \in M^\eps(\calO)$
is determined by its values on $x \in \calF_\eps$.  Hence we think of
$\calF_\eps$ as a fundamental domain for functions in $M^\eps(\calO)$.

\begin{prop}\label{prop42}
 Suppose $\phi \in M^\eps(\calO)$ is a degree $d$ eigenform.
Then there are at most $\# \calF_\eps - d$ elements
$x \in \calF_\eps$ such that $\phi(x) = 0$.
\end{prop}

\begin{proof}
The $d$ Galois conjugates $\phi_1, \dots, \phi_d$ of $\phi$ generate a
$d$-dimensional subspace of $M^\eps(\calO)$.
On the other hand, since each $\phi_i$ has the same support as $\phi$, 
the dimension of the span of $\phi_1, \dots, \phi_d$
is at most the number of $x \in \calF_\eps$ such that $\phi(x) \ne 0$.
\end{proof}

Since each $\eps$-admissible orbit has size at most $2^{\omega(N)}$
ideal classes, we conclude the following.

\begin{cor} \label{cor43}
Suppose $\phi \in M^\eps(\calO)$ is a degree $d$ eigenform.  Then
$\phi$ has at most $2^{\omega(N)} (\dim M^\eps(\calO) - d)$ nontrivial zeroes.
\end{cor}

\subsection{Conditional results}

Note that \cref{gal-conj} for $k=2$ and fixed $r$ implies that, as $N$
ranges over $\Sq_r$, 100\% of Atkin--Lehner eigenspaces $S_2^{\new, \eps}(N)$ 
are spanned by a single Galois orbit, and thus 100\% of
newforms lie in such Atkin--Lehner eigenspaces.
Consequently, the following implies \cref{thm14}.

\begin{thm} \label{thm:nozeroes} Let $\eps$ be a sign pattern for $N$.
Suppose the Atkin--Lehner eigenspace $S_2^{\new, -\eps}(N)$
is spanned by a single Galois orbit.  Then no eigenforms in $M^\eps(\calO)$ have
nontrivial zeroes.
\end{thm}

\begin{proof}
If $\eps \ne +_N$, then it is immediate from \cref{cor43} that any of the Galois conjugate
eigenforms in $M^\eps(\calO) = S^\eps(\calO)$ have no trivial zeroes.
Also, $\one$ clearly has no trivial zeros.

So assume $\eps = +_N$, $\dim S^{+_N}(\calO) > 0$ and consider the eigenforms 
$\phi_1, \dots, \phi_d \in S^{+_N}(\calO)$,
which are all Galois conjugate by assumption.
By the same reasoning as in the proof of \cref{prop42}, each $\phi_i$
must be zero-free if there exists some zero-free (not necessarily eigen)
form $\phi \in S^{+_N}(\calO)$.  

But the existence of such a $\phi$ is easy to see.
Recall $X_1, \dots, X_t$ are the orbits of $\Cl(\calO)$ under $\mathrm{Pic}(\calO)$.
These are all $+_N$-admissible, and because $\dim S^{+_N}(\calO) > 0$,
i.e., $\dim M^{+_N}(\calO) > 1$, we have $t \ge 2$.
Note $\phi \in S^{+_N}(\calO)$
if and only if $\phi$ is constant on each $X_j$ and
$\sum_{i=1}^h e_i^{-1} \phi(x_i) = 0$.  Put $e_{X_j} = e_i$
for any $x_i \in X_j$---this does not depend on the choice of
$x_i$---see \cite[Lemma 4.1]{me:cong2}.
Then we may construct a
zero-free $\phi \in S^{+_N}(\calO)$ by choosing nonzero $a_j$'s
such that $\sum_{j=1}^t e_{X_j}^{-1} a_j = 0$
and defining $\phi = \sum a_j 1_{X_j}$.  
\end{proof}


\section{Data and conjectures}
\label{sec:data}


In this section, we present some data and conjectures on zeroes
of quaternionic modular forms.  Our data is based on calculations
of quaternionic modular forms in both Sage \cite{sage} 
and Magma \cite{magma}.\footnote{Some of our code is available at \url{https://math.ou.edu/~kmartin/data/}}

\subsection{Computations}
We calculate a basis of eigenforms $\phi$ for $M(\calO)$ as follows.
We use built-in functions to compute Brandt matrices $T_p$ acting on 
$M(\calO)$ for prime $N$ with Sage or general $N \in \Sq_\odd$ with Magma.
Given $N$, we compute $T_p$ for $p \nmid N$ for all small $p$ until we find
a $T_p$ with no repeated eigenvalues.  (In all cases, we find such a $p$
and it is usually the smallest $p \nmid N$.)   
Then the eigenvectors of $T_p$ give a basis of eigenforms for $M(\calO)$ as
in \cref{ex:37},
and we may scale them to make all entries algebraic integers.
We compute exact eigenvectors of $T_p$ whenever $N < 4000$ ($N \in \Sq_\odd$)
or when the degree $d$ of the (minimal polynomial of the) eigenvalue is small and 
$N < 20000$ is prime. 

Given the exact eigenvectors, we can determine the exact number and position of
zeroes of our eigenforms.  Then for prime level $N$, we determine which zeroes are
trivial zeroes as follows.  For small $N$, one can compute $T_N$ and determine the sign of an eigenform $\phi$ and identify any trivial zeroes, but the calculation of $T_N$
is slow for $N$ large.  (We expect that one should be able to implement a faster algorithm 
using $\sigma_N$, but we have not attempted this.)
Instead, suppose we have computed a basis of eigenforms for $S(\calO)$.  
Recall for prime $N$, the trivial zeroes occur exactly at the fixed 
points of $\sigma_N$ for $\phi \in S^{-_N}(\calO)$.  We know the number $r$ of
fixed points of $\sigma_N$ from \cref{prop:nfixedpoints}.
Then we find the $x_i$'s in $\Cl(\calO)$ such that at least $\dim S^{-_N}(\calO)$ eigenforms
are zero at $x_i$.  In our calculations, there are always exactly $r$ such $x_i$'s, 
so these $x_i$'s must be the fixed points of $\sigma_N$, and hence the sign
(i.e., eigenvalue of $T_N$) for an eigenform $\phi$ is $+1$ if $\phi$ is nonzero at
one of these $x_i$'s and $-1$ otherwise.

When $d$ and $N$ are both large, the implementations of exact eigenvector calculations 
are slow in Magma and very slow in Sage.   For example, consider $T_2$ with $N=1009$, which is represented by an $84 \times 84$ matrix.  Its characteristic polynomial has irreducible factors of degrees 1, 37 and 46.  Here exact
calculations of eigenspaces took about 57 minutes of CPU time in Sage,
and about 2 minutes of CPU time in Magma.
Instead, we use  {\tt numpy} to compute approximate eigenvectors for prime $N < 20000$, and thus compute numerical zeroes of eigenforms.  (For $N=1009$,
 computing numerical eigenvectors 
with {\tt numpy} only required about 0.09 seconds of CPU time.)

The issue, of course, with numerical eigenvector methods is that one needs to worry about
how close the set of numerical zeroes matches the set of 
actual zeroes.
While we can check that these counts do not always match exactly, they do appear to be
very close.  More precisely, for prime $N < 20000$, we used a numerical version of the algorithm described above
to numerically compute the fixed points of eigenforms, and thus numerically compute
the root number of such eigenforms.  Because this method always resulted in the
correct number of fixed points, we believe that our numerical methods are properly
detecting every trivial zero.  In addition, comparing the numerical calculations with
exact calculations for both small degree forms with prime $N < 20000$ 
and all forms with prime $N < 4000$ suggests that the number of numerical nontrivial
zeroes we find is quite also close to the actual number of nontrivial zeroes.
(Most numerical zeroes are accounted for by trivial zeroes or by nontrivial zeroes
of small degree forms.)  As further evidence that we are not picking up too many
 numerical nontrivial zeros which are not actual zeroes, consider 
the 1563 prime levels $N < 20000$ such that $S(\calO)$ has exactly 2 Galois orbits.
The eigenforms have no nontrivial zeroes in these levels by \cref{thm:nozeroes},
and only 20 of these levels possess numerical nontrivial zeroes, and in each of these
20 levels there are exactly 2 spurious numerical nontrivial zeroes.
(The smallest such level is 7351, and the rest are over 10000.)

We remark that in \cite{me:maeda}, we in fact computed a $T_p$ with
no repeated eigenvalues for each prime $N < 60000$.  In principle we could use this
to extend our calculations of numerical zeroes for $N > 20000$, 
but the issue is that we do not have a good, practical way to either improve or
estimate the accuracy of this numerical approach for a much larger range of $N$.  
The problem is that numerical
methods are not good at distinguishing eigenspaces with close eigenvalues.
In large levels it happens the numerical approximations for some nonzero entries 
are smaller (in absolute value) than the numerical approximations for entries which 
are actually zeroes, so one cannot simply adjust a numerical threshold to only detect 
true zeroes.  Moreover, there is no simple way
to increase precision in the numerical eigenvector methods in {\tt numpy}, 
and in any case increasing the precision would significantly slow down the
numerical methods.

For squarefree levels $N$, it is more complicated to determine which zeroes are trivial.
Again, one could compute $T_p$ for each $p | N$ and use this to construct the signed
graphs $\Sigma^\eps$, which is feasible when all prime factors of $N$ are small.
On the other hand, we have criteria to determine  which sign patterns $\eps$ have
trivial zeroes, and can compute the dimension of each $M^\eps(\calO)$ from the
formulas for dimensions of Atkin--Lehner spaces in \cite{me:dim}.  Thus
we can count how many forms have no trivial zeroes.  We use this to determine exactly
how many forms in $S(\calO)$ with no trivial zeroes are zerofree for squarefree levels
$N < 4000$.

\subsection{Counting zeroes} \label{sec52}
For $N \in \Sq_\odd$, fix a maximal quaternionic order $\calO=\calO_N$ of level
$N$ in $B_N$.
Let $S$ be a Hecke-invariant subspace of $M(\calO_N)$.
By the number of (trivial/nontrivial) zeroes for $S$ we mean
the total number of (trivial/nontrivial) zeroes as we run over a basis of eigenforms
for $S$.

\begin{prop} As $N \to \infty$ in $N \in \Sq_\odd$, the number of trivial zeroes for
$S(\calO_N)$ is $O(N^{3/2+\epsilon})$ for any $\epsilon > 0$.  This is essentially
optimal in the sense that this number is not $O(N^{3/2-\epsilon})$ for any
$\epsilon > 0$ when $N$ ranges over all of $\Sq_\odd$.
\end{prop}

\begin{proof} 
By \cref{triv0-bound}, the maximum number of trivial zeroes of an eigenform
$\phi \in S(\calO_N)$ grows like $O(N^{1/2+\epsilon})$.  Since $\dim S(\calO_N)
\sim \frac{\phi(N)}{12} = O(N)$, using this upper bound on every eigenform (up to scaling) 
in $S(\calO_N)$ gives the first statement.

We know that the number of trivial zeroes cannot be $O(N^{3/2-\epsilon})$ for any
$\epsilon > 0$ because \cref{triv0-prime-lb} tells us it is not when we restrict to
the subsequence of $N \in \Sq_1$.
\end{proof}

For comparison, the total number of values (with multiplicity)
of a basis of eigenforms for $S(\calO_N)$ is $O(N^2)$.
Precisely, there are $h-1$ linearly independent eigenforms in $S(\calO_N)$,
each of which have $h$ (not necessarily distinct) values $\phi(x_1), \dots,
\phi(x_h)$, giving  $h^2-h$ total values 
where $h = 1 + \dim S_2^\new(N) \sim \frac{\phi(N)}{12}$.

We expect that most zeroes are trivial zeroes.  Quantitatively,
we predict the following.

\begin{conj} \label{conj:nnontriv}
As $N \to \infty$ along $\Sq_r$ for some $r$ or along $\Sq_\odd$,
the number of nontrivial zeroes for
$S(\calO_N)$ is $O(N^{1+\epsilon})$ for any $\epsilon > 0$.
\end{conj}

This together with \cref{thm14} suggests \cref{conj11}(i).

Before we present some heuristics for this, we present some data.
In \cref{ntriv-max}, we plot the
total number of numerical nontrivial zeroes of $S(\calO_N)$, for prime levels $N < 20000$.  
In \cref{ntriv-graph} we also graph the total
number of numerical nontrivial zeroes for all prime levels $\le X$ when
$X < 20000$.  
These graphs indeed suggest that the number of nontrivial
zeroes for $S(\calO_N)$ grows not much faster than linearly in $N$, and that
the same is true for the total number of nontrivial zeroes of all prime levels up to $X$.  
The latter quantity should not grow too much faster than the former quantity
as we expect 100\% of prime levels should have no nontrivial zeroes from 
\cref{thm14}.
We note that out of the 2262 prime levels $N < 20000$, there are numerical nontrivial
zeroes in 682 of these levels.

\begin{figure}
\begin{minipage}{.5\textwidth}
\captionof{figure}{Number of nontrivial zeroes in each prime level $N$}
\resizebox{.9\linewidth}{!}{
\begin{tikzpicture}
    \begin{axis}[xlabel=$N$, scaled x ticks = false, scaled y ticks = false]
    \addplot[only marks, mark size = 0.7] table {qmf-zeroes-tots.dat};
    \end{axis}
\end{tikzpicture}
}
 \label{ntriv-max}
\end{minipage}%
%
\begin{minipage}{.5\textwidth}
\captionof{figure}{Number of nontrivial zeroes for prime levels $\le X$}
\resizebox{.9\linewidth}{!}{
\begin{tikzpicture}
    \begin{axis}[xlabel=$X$, scaled x ticks = false, scaled y ticks = false]
    \addplot[]
        plot coordinates {      
  (3, 0)
(37, 0)
(79, 2)
(131, 7)
(181, 33)
(239, 81)
(293, 115)
(359, 213)
(421, 274)
(479, 435)
(557, 602)
(613, 810)
(673, 925)
(743, 1163)
(821, 1325)
(881, 1413)
(953, 1507)
(1021, 1741)
(1091, 1963)
(1163, 2083)
(1231, 2317)
(1301, 2604)
(1399, 3046)
(1459, 3330)
(1531, 3567)
(1601, 3823)
(1667, 4317)
(1747, 4574)
(1831, 4828)
(1907, 5295)
(1997, 5879)
(2069, 6247)
(2137, 6723)
(2237, 7321)
(2297, 7887)
(2377, 8734)
(2441, 9085)
(2543, 9151)
(2633, 10015)
(2693, 10148)
(2753, 10611)
(2837, 11031)
(2917, 11461)
(3011, 11828)
(3089, 12462)
(3191, 12844)
(3271, 13922)
(3347, 14531)
(3449, 14932)
(3527, 15547)
(3583, 16191)
(3671, 16508)
(3739, 16692)
(3833, 17202)
(3917, 17882)
(4003, 18879)
(4079, 19457)
(4157, 19873)
(4243, 20397)
(4337, 20812)
(4423, 21489)
(4513, 23151)
(4597, 23323)
(4673, 23469)
(4783, 23727)
(4871, 24376)
(4951, 24718)
(5011, 24907)
(5101, 25345)
(5197, 25841)
(5297, 26351)
(5399, 27904)
(5471, 29701)
(5531, 30283)
(5647, 30525)
(5711, 30936)
(5807, 31917)
(5867, 32608)
(5981, 32796)
(6073, 35108)
(6151, 35270)
(6247, 36404)
(6317, 36884)
(6379, 37678)
(6491, 38087)
(6581, 39286)
(6689, 39720)
(6779, 40376)
(6857, 40740)
(6949, 41477)
(7013, 42093)
(7121, 44403)
(7213, 44889)
(7309, 45609)
(7433, 45923)
(7517, 46370)
(7577, 46637)
(7669, 47488)
(7741, 48605)
(7853, 49533)
(7933, 50899)
(8053, 50899)
(8123, 51693)
(8231, 52337)
(8297, 52957)
(8419, 54867)
(8521, 55862)
(8609, 57175)
(8689, 57942)
(8753, 59765)
(8839, 60593)
(8941, 61535)
(9029, 63349)
(9133, 64608)
(9209, 65776)
(9311, 66909)
(9397, 67471)
(9463, 67847)
(9547, 69105)
(9649, 69709)
(9743, 70133)
(9829, 72477)
(9907, 73977)
(10037, 75591)
(10111, 77847)
(10193, 78159)
(10289, 78971)
(10369, 81650)
(10477, 82328)
(10597, 85480)
(10667, 85827)
(10771, 87544)
(10867, 88780)
(10973, 89687)
(11071, 92195)
(11161, 92573)
(11261, 94043)
(11353, 95752)
(11467, 96534)
(11551, 97164)
(11681, 97164)
(11783, 98012)
(11863, 98312)
(11941, 100761)
(12041, 102242)
(12119, 104222)
(12239, 107036)
(12323, 107828)
(12413, 108252)
(12497, 108254)
(12577, 110627)
(12653, 112174)
(12757, 113440)
(12853, 114280)
(12953, 114622)
(13033, 114622)
(13127, 117023)
(13219, 117023)
(13327, 117353)
(13421, 117353)
(13523, 118719)
(13633, 119729)
(13711, 121128)
(13799, 122946)
(13901, 123200)
(13999, 123684)
(14087, 124300)
(14221, 126537)
(14341, 127963)
(14431, 130021)
(14537, 130665)
(14627, 131474)
(14717, 131476)
(14779, 131740)
(14869, 133159)
(14957, 134616)
(15083, 136016)
(15173, 138347)
(15269, 140809)
(15331, 142656)
(15427, 144604)
(15527, 144604)
(15629, 145581)
(15727, 149616)
(15791, 151220)
(15889, 153244)
(15991, 154839)
(16087, 156709)
(16189, 157247)
(16301, 157699)
(16417, 160065)
(16493, 160454)
(16619, 160456)
(16699, 161212)
(16829, 162871)
(16927, 166073)
(17021, 168329)
(17107, 169812)
(17207, 170292)
(17327, 172787)
(17401, 175913)
(17489, 179117)
(17581, 180700)
(17683, 182399)
(17791, 185301)
(17909, 187015)
(17981, 188039)
(18077, 190500)
(18169, 193178)
(18253, 194436)
(18341, 196013)
(18439, 198232)
(18523, 199536)
(18671, 200230)
(18773, 200678)
(18913, 202180)
(19031, 202184)
(19141, 203679)
(19237, 204237)
(19373, 204609)
(19429, 207109)
(19489, 208279)
(19583, 209315)
(19717, 209317)
(19813, 210162)
(19919, 210840)
(19997, 212516)      
        };
\end{axis}
\end{tikzpicture}
}
 \label{ntriv-graph}
 \end{minipage}
\end{figure}

In the rest of this section, we explain our main reasons for believing
in \cref{conj:nnontriv}, which are interwined with \cref{conj12}.

Say $\phi \in S(\calO)$ has rationality field $K=K_\phi$ of degree $d$.
Normalize $\phi$ so that $\phi$ is integral and primitive, i.e.,
$\phi(x_1), \dots \phi(x_h)$ all lie in $\frako_K$ and have no common
factors besides units.  Since $K$ is totally real, we have
$d$ real embeddings, $\sigma_1, \dots, \sigma_d: K \to \R$.  
This allows us to realize $\frako_K$ as a lattice $\Gamma \subset \R^d$ via
$\alpha \mapsto (\sigma_1(\alpha), \dots, \sigma_d(\alpha))$.
Similarly, let $\Lambda = \Gamma^h \subset \R^{dh}$ be the lattice obtained by 
a component-wise embedding.  
Then we may view $\phi \in \Lambda$
as the lattice point corresponding to $(\phi(x_1), \dots \phi(x_h)) \in \frako_K^h$.

For convenience, we will take the metric on $\R^{dh}$ to be the product of
$\frac 1{\sqrt{e_i}}$ times the standard metric on $\R^d$ over $1 \le i \le h$.  Let
$\lambda = \prod e_i^{-1/2}$.  Then $\Lambda \subset \R^{dh}$ is a lattice with volume 
$\lambda \Delta_K^{h/2}$ where $\Delta_K$ is the discriminant of $K$.
Then the square length of $\phi \in \Lambda$ is
\[ \langle \phi, \phi \rangle_\Lambda := \sum_{\sigma_j} \sum_{i=1}^h \frac 1{e_i} 
\sigma_j(\phi(x_i))^2 = \tr_{K/\Q} (\phi, \phi). \]
We would like to model $\phi$ as a random lattice point in $\Lambda$ on the sphere of
radius $\langle \phi, \phi \rangle_\Lambda^{1/2}$.

For a hypersphere $S^{n-1}$ in $\R^{n}$, and coordinate functions $X_1, \dots, X_d$ 
of random point on $S^{n-1}$, a result of Poincar\'e tells us
that the distribution of $(X_1, \dots, X_d)$ tends to a $d$-dimensional normal distribution
as $n \to \infty$---in fact $d$ is also allowed to vary,
provided that $d$ grows slower than $n$, i.e., $d=o(n)$ \cite{DF}.
Consequently, viewing $\phi \in \Lambda$ as a randomly chosen lattice point on some hypersphere of radius $R$ in $\R^{dh}$, a given value 
$\phi(x_i) \in \frako_K \subset \R^d$  should be distributed approximately like a 
$d$-dimensional normal distribution when $N$ is large. Note that $d \le h$ so $d=o(n)$ where $n=dh$.  (See \cref{fig:deg1vals,fig:deg2vals} below for 
two numerical examples when $d=1, 2$.) 
As long as the radius $R$ is not too small with respect to $\Delta_K$, the probability that 
$\phi(x_i) = 0$ will decrease as $d$ increases, essentially because the number of lattice
points $\frako_K$ in the ball of radius $R$ is increasing.
While we do not have precise heuristics about 
$\langle \phi, \phi \rangle_\Lambda = R^2$, but numerically it appears
to increase suitably rapidly with $d$ and $\Delta_K$.

This idea, together with data and \cref{cor43}, is what leads us to believe in \cref{conj12},
i.e., that forms with smaller degree are more likely to have zeroes.  
In \cref{tab:deg-hist} we tabulate the nontrivial exact zeroes by degree
among prime levels $N < 4000$, which supports this conjecture.  
The first row of data tells us how many Galois orbits of eigenforms of degree $d$
there are in this range, and the second row counts how many of these orbits 
have nontrivial zeroes.
The third row counts the total number of nontrivial zeroes coming from degree $d$ forms
in this range, the fourth row gives the proportion of all zeroes which are nontrivial zeroes,
and the final row gives the proportion all values (with multiplicity) which are nontrivial
zeroes.
(In the third row, we count each eigenform in a Galois orbit separately; in the last
two rows the proportion is the same looking at individual eigenforms or looking at
Galois orbits.)  While this range $N < 4000$ is rather limited,
the key point of this table is that it indicates that the frequency of nontrivial zeroes
tends to decrease in $d$.

\begin{table}
\caption{Degree histograms about nontrivial zeroes for prime levels $N < 4000$}
\begin{tabular}{r|rrrrrrrrrr}
$d$ & 1  &  2  &  3  &  4  &  5  &  6  &  7  &  8  &  9  & $\ge 10$ \\
\hline
tot \#orbits & 179 & 133 & 57 & 25 & 19 & 10 & 18 & 3 & 12 & 983 \\
\#orbits w/nontriv 0's & 152  &  110  &  35  &  14  &  5  &  2  &  2  &  0  &  2 & 0  \\
\#nontriv 0's & 9730 & 5896 & 1860 & 864 & 210 & 120 & 28 & 0 & 27 & 0 \\
\#nontriv 0's/all 0's & 0.873 & 0.713 & 0.543 & 0.451 & 0.214 & 0.175&  0.035& 0 & 0.032 & 0 \\
\#nontriv 0's/all vals & 0.396  & 0.177  &  0.123  &  0.101  &  0.046  &  0.042  &  0.007  &  0  &  0.006  &  0   
\end{tabular} \label{tab:deg-hist}
\end{table}

Moreover, from \cite{me:maeda}, we have the expectation that 100\% of the time
that an eigenspace $S^\eps(\calO_N)$ has multiple Galois orbits it is because of the
existence of degree 1 forms.  In summary, most of the small Galois orbits
consist of degree 1 forms, and degree 1 forms are more likely than higher degree forms to
have zeroes.  This leads to the following more precise conjecture.

\begin{conj} \label{conj:deg1}
As $N \to \infty$ along $\Sq_r$ for some odd $r$, or along $\Sq_\odd$,
$100\%$ of nontrivial zeroes of $S(\calO_N)$ come from degree $1$ forms. 
\end{conj}

\begin{rem} \label{rem:deg2}
We similarly expect that if one restricts to forms of degree
$\ge 2$, that 100\% of non-trivial zeroes come from degree 2 forms.
More generally, we expect a similar sort of phenomenon that degree $d$ forms
tend to contribute most of the zeroes among degree $\ge d$ forms.
However, we are hesitant to formulate a precise conjecture about this in general, 
as it is not even clear if for fixed $d$ there should be infinitely
many forms of degree $d$ with prime level.  See \cite{me:maeda} for
a brief discussion of this latter issue.
\end{rem}

In \cref{d1rat-graph}, we present some data on the proportion of numerical
nontrivial zeroes of prime levels $\le X$ which are accounted for by the (actual) nontrivial
zeroes of degree 1 forms.  Note that this proportion appears to be slowly increasing.
While it is not clear from the graph alone whether this proportion is tending to 1,
we expect that the slow rate of increase is due to the existence of many small degree
forms in small levels (cf.\ \cite{me:maeda}).  In addition, we believe the numerical
nontrivial zeroes may be a slight overcount of the actual nontrivial zeroes for $N$ large.

\begin{figure}
\begin{minipage}{.5\textwidth}
\captionof{figure}{Proportion of nontrivial zeroes for prime levels $\le X$ coming from degree 1 forms}
\resizebox{.9\linewidth}{!}{
\begin{tikzpicture}
    \begin{axis}[xlabel=$X$, ymin=0, ymax=1.1, scaled x ticks = false, scaled y ticks = false]]
    \addplot[]
        plot coordinates {     
(73, 1.0)
(83, 1.0)
(101, 1.0)
(113, 1.0)
(139, 1.0)
(179, 0.8181818181818182)
(229, 0.43209876543209874)
(269, 0.47126436781609193)
(307, 0.630057803468208)
(347, 0.6038647342995169)
(359, 0.6150234741784038)
(389, 0.6178861788617886)
(433, 0.5479041916167665)
(467, 0.5287356321839081)
(557, 0.4833887043189369)
(571, 0.4878048780487805)
(643, 0.45823389021479716)
(677, 0.4440961337513062)
(709, 0.4583333333333333)
(739, 0.46345657781599314)
(811, 0.4505660377358491)
(829, 0.45428365730741543)
(1019, 0.43365881677197016)
(1091, 0.41721854304635764)
(1171, 0.4145437702640111)
(1259, 0.46197874080130824)
(1297, 0.46870109546165883)
(1373, 0.42678923177938277)
(1439, 0.40950639853747717)
(1531, 0.4087468460891506)
(1607, 0.4076269002834321)
(1621, 0.4282603435760948)
(1747, 0.4188893747267162)
(1867, 0.4141248720573183)
(1901, 0.4242541650523053)
(1913, 0.4379060701689252)
(1949, 0.4514492753623188)
(1979, 0.467744770609949)
(2017, 0.4743740547807091)
(2089, 0.4630373345232783)
(2143, 0.4701306326141201)
(2213, 0.48154457694491765)
(2237, 0.49241906843327415)
(2269, 0.48822242244819153)
(2341, 0.45591880089696685)
(2357, 0.4661308840413318)
(2539, 0.4643208392525407)
(2609, 0.45676711193952574)
(2699, 0.4630640988789882)
(2797, 0.4607070892104781)
(2843, 0.46415498616194983)
(2939, 0.465698980473475)
(3023, 0.46937416777629826)
(3109, 0.4577106518282989)
(3181, 0.4676113360323887)
(3229, 0.4712626338814301)
(3257, 0.47776205021232215)
(3313, 0.4784102060843965)
(3347, 0.48792237285802764)
(3391, 0.49966514867398876)
(3469, 0.501585657886221)
(3623, 0.5050278652774413)
(3797, 0.5040042397833)
(3851, 0.5037633163501621)
(3877, 0.5123463814292216)
(3931, 0.5131607929515418)
(3967, 0.5193488123832399)
(4021, 0.5205680448566787)
(4139, 0.5219194013247712)
(4229, 0.5180279617365711)
(4289, 0.5181234615570249)
(4339, 0.5244355530151472)
(4451, 0.5203016508024095)
(4481, 0.5351780065053692)
(4507, 0.5315969072610255)
(4603, 0.5379436703736844)
(4799, 0.543567443386938)
(5077, 0.5317814164529493)
(5197, 0.5365891412871019)
(5303, 0.5374562163383677)
(5393, 0.5175602064220184)
(5419, 0.5247112683383623)
(5503, 0.5173029251363411)
(5651, 0.5215179968701096)
(5689, 0.5232738557020946)
(5741, 0.5252726410621148)
(5867, 0.5235525024533857)
(6011, 0.5329116483262571)
(6067, 0.5393642474649653)
(6199, 0.5410753880266075)
(6323, 0.5403582604942269)
(6427, 0.5398351648351648)
(6571, 0.537638006022081)
(6691, 0.5383373086861459)
(6899, 0.5398858313817331)
(6967, 0.5380058525444293)
(7057, 0.5494922981280652)
(7187, 0.5498350427542249)
(7219, 0.5541660222276232)
(7451, 0.5520454151499393)
(7669, 0.545822102425876)
(7699, 0.5513956487374683)
(7757, 0.5541555068582096)
(7867, 0.5589610806687942)
(8101, 0.5513187347931874)
(8219, 0.5534898828744483)
(8243, 0.5587174500066091)
(8419, 0.5566187325714911)
(8539, 0.5514345820001061)
(8597, 0.5562046348928728)
(8699, 0.5593400198541745)
(8731, 0.5595195744608765)
(8803, 0.56802220863559)
(8929, 0.5712683838465914)
(9011, 0.5651075786515967)
(9127, 0.5692793462109955)
(9161, 0.5740786090363659)
(9277, 0.5799610571915895)
(9341, 0.5778037971869395)
(9467, 0.5822309565013346)
(9479, 0.5841641971884443)
(9539, 0.5877867013964257)
(9587, 0.590959348659554)
(9811, 0.5925328288448698)
(9901, 0.591862890800758)
(9941, 0.5906363831154917)
(10061, 0.5907074695623561)
(10091, 0.5944282136894825)
(10163, 0.5972483763114408)
(10331, 0.5950258389810262)
(10357, 0.6029883649724433)
(10499, 0.6004463482236061)
(10597, 0.600421151146467)
(10691, 0.6033261329059532)
(10789, 0.5996540107438769)
(10859, 0.6037846361793197)
(10957, 0.6071894477460501)
(10987, 0.610382429054876)
(11059, 0.6075926026357178)
(11171, 0.6107291565798393)
(11197, 0.6121799662616643)
(11321, 0.618917728026186)
(11467, 0.6209832805022065)
(11731, 0.6183324490878668)
(11867, 0.6191246990267689)
(11923, 0.6220308013573479)
(11971, 0.62765305113243)
(12011, 0.6317755912443027)
(12101, 0.630423519026693)
(12197, 0.6357498155958617)
(12211, 0.636777432285618)
(12277, 0.6389013035381751)
(12413, 0.6409673724272993)
(12553, 0.6374766812163802)
(12763, 0.6330765451163444)
(13093, 0.6339983625630198)
(13451, 0.6323757796919421)
(13537, 0.6350259206217599)
(13649, 0.63778644036804)
(13723, 0.6361501194801794)
(14107, 0.6289779767376794)
(14173, 0.629001793941693)
(14389, 0.6224354118029035)
(14461, 0.6237389405951971)
(14891, 0.6168861800184197)
(15013, 0.6167753408417309)
(15091, 0.6157861552222174)
(15131, 0.6152997372127125)
(15193, 0.6171527523298341)
(15299, 0.6148187889141453)
(15391, 0.611957909933922)
(15619, 0.6123257842712991)
(15643, 0.6165597683418581)
(15773, 0.6194023914945516)
(15859, 0.6170053412696849)
(15889, 0.6195870637675863)
(15973, 0.6198583026020912)
(16061, 0.6195240860448347)
(16193, 0.6208258345151259)
(16363, 0.6234646256054664)
(16411, 0.6228907006528598)
(16649, 0.6253351539225422)
(16811, 0.6270856078409602)
(16883, 0.6308200170496798)
(16937, 0.6324454653959419)
(16987, 0.631828146071087)
(17299, 0.6337127673153015)
(17333, 0.635166918040555)
(17389, 0.6354561629896597)
(17483, 0.6350301936091596)
(17573, 0.6312230215827338)
(17681, 0.6307271421444196)
(17827, 0.6315755609769175)
(17989, 0.6295957631239562)
(18097, 0.6339997284283312)
(18149, 0.637220594477632)
(18269, 0.64054469918449)
(18289, 0.641550351120363)
(18379, 0.6440047748457064)
(18523, 0.6460037286504691)
(18859, 0.6439779411035399)
(19079, 0.6458188419738835)
(19211, 0.6484721181764322)
(19403, 0.6511029785012256)
(19583, 0.6490886940735255)
(19927, 0.6495650940720431)        
        };
\end{axis}
\end{tikzpicture}
}
 \label{d1rat-graph}
\end{minipage}%
\begin{minipage}{.5\textwidth}
\captionof{figure}{Proportion of values of degree 1 forms (excluding trivial zeroes) which are 0}
\resizebox{.9\linewidth}{!}{
\begin{tikzpicture}
    \begin{axis}[xlabel=$X$, ymin=0, ymax=1.1, scaled x ticks = false, scaled y ticks = false]]
    \addplot[]
        plot coordinates { 
        (11, 0.0)
(19, 0.0)
(43, 0.0)
(61, 0.0)
(73, 0.06896551724137931)
(83, 0.06060606060606061)
(101, 0.044444444444444446)
(113, 0.109375)
(139, 0.16666666666666666)
(179, 0.25471698113207547)
(229, 0.26515151515151514)
(269, 0.25)
(307, 0.3784722222222222)
(347, 0.3787878787878788)
(359, 0.3411458333333333)
(389, 0.3431151241534989)
(433, 0.34398496240601506)
(467, 0.35276073619631904)
(557, 0.34437869822485206)
(571, 0.3437815975733064)
(643, 0.3395225464190981)
(677, 0.33650039588281866)
(709, 0.35047067342505434)
(739, 0.3583776595744681)
(811, 0.37080745341614907)
(829, 0.3660092807424594)
(1019, 0.3619367209971237)
(1091, 0.3685868586858686)
(1171, 0.3560063643595863)
(1259, 0.36428110896196003)
(1297, 0.36106088004822184)
(1373, 0.36692068868190797)
(1439, 0.35887850467289717)
(1531, 0.36477358018513883)
(1607, 0.3744378698224852)
(1621, 0.38228941684665224)
(1747, 0.3908608731130151)
(1867, 0.39175058094500387)
(1901, 0.39970797590801244)
(1913, 0.38552050988723646)
(1949, 0.3888888888888889)
(1979, 0.3961000297707651)
(2017, 0.4014505119453925)
(2089, 0.3867080745341615)
(2143, 0.38272195005376985)
(2213, 0.38916934373565853)
(2237, 0.39707016191210487)
(2269, 0.40175828831691557)
(2341, 0.393741718479258)
(2357, 0.39866457187745485)
(2539, 0.4018346888594666)
(2609, 0.40268395106810295)
(2699, 0.40871035940803385)
(2797, 0.4100766345997065)
(2843, 0.4080527431127855)
(2939, 0.4074074074074074)
(3023, 0.4122204356088291)
(3109, 0.40595036661026507)
(3181, 0.40871044572984006)
(3229, 0.41064738744659873)
(3257, 0.40854940434477927)
(3313, 0.40106951871657753)
(3347, 0.40391955791032874)
(3391, 0.4063061591243261)
(3469, 0.4094679558302953)
(3623, 0.4147761194029851)
(3797, 0.4133983771251932)
(3851, 0.4091892400300978)
(3877, 0.411563071297989)
(3931, 0.41395700071073205)
(3967, 0.4205021824625092)
(4021, 0.4177811422323156)
(4139, 0.4223404255319149)
(4229, 0.4153392330383481)
(4289, 0.4106575876974867)
(4339, 0.40982691233947516)
(4451, 0.40996376811594204)
(4481, 0.4137444023424044)
(4507, 0.41387543718052194)
(4603, 0.41412451617135737)
(4799, 0.4146066712560235)
(5077, 0.4110524871145811)
(5197, 0.41201640221073277)
(5303, 0.4134913505838718)
(5393, 0.4085315832649713)
(5419, 0.40832905994440094)
(5503, 0.4078436482084691)
(5651, 0.40721977546396476)
(5689, 0.40257640944020295)
(5741, 0.4040313168644233)
(5867, 0.4058866883811607)
(6011, 0.4101771780450392)
(6067, 0.4117148261691997)
(6199, 0.4115266242253046)
(6323, 0.4131259797046448)
(6427, 0.41273175263561823)
(6571, 0.41286561264822136)
(6691, 0.4151084374281554)
(6899, 0.41515344788774666)
(6967, 0.41544651222603934)
(7057, 0.41238271327607956)
(7187, 0.4114643690902067)
(7219, 0.4132437019096107)
(7451, 0.41177896659285956)
(7669, 0.4105098113745427)
(7699, 0.41183141476751955)
(7757, 0.41216063263057817)
(7867, 0.41211726577183194)
(8101, 0.4115365056302216)
(8219, 0.41289071965107826)
(8243, 0.4136214946738614)
(8419, 0.414979482022991)
(8539, 0.415851219837355)
(8597, 0.41619442735803375)
(8699, 0.4161890935024579)
(8731, 0.4129863959187756)
(8803, 0.4077505219206681)
(8929, 0.40991405949368565)
(9011, 0.41025670410268167)
(9127, 0.411147255105805)
(9161, 0.41224830739470675)
(9277, 0.41527154823020807)
(9341, 0.41437256861036115)
(9467, 0.41524590506871956)
(9479, 0.4117526006564664)
(9539, 0.4120618818158762)
(9587, 0.41103209955456127)
(9811, 0.41234837997696056)
(9901, 0.4073590961130609)
(9941, 0.4077964346475331)
(10061, 0.40891654745742634)
(10091, 0.40899545919557767)
(10163, 0.4093758780556336)
(10331, 0.4093472657128523)
(10357, 0.4077417431344619)
(10499, 0.4085342955573334)
(10597, 0.4102637889688249)
(10691, 0.4094452194610613)
(10789, 0.4096503572622593)
(10859, 0.4092002106918479)
(10957, 0.4101913226875565)
(10987, 0.4101359887047633)
(11059, 0.41065171175133786)
(11171, 0.4107486576841252)
(11197, 0.4094401599542988)
(11321, 0.41059869987054387)
(11467, 0.41184149056033414)
(11731, 0.4108968622026957)
(11867, 0.4097184098121393)
(11923, 0.40943123170352946)
(11971, 0.40889146612544697)
(12011, 0.4109738950074122)
(12101, 0.41278616842158167)
(12197, 0.41193375040594116)
(12211, 0.40895010676370636)
(12277, 0.4079669431314843)
(12413, 0.4077236321755328)
(12553, 0.40872326960035765)
(12763, 0.409070035094085)
(13093, 0.40968552128175)
(13451, 0.41012488235520744)
(13537, 0.40993258256878257)
(13649, 0.4099585948978229)
(13723, 0.4080127649970344)
(14107, 0.4081122435747815)
(14173, 0.4078963552129884)
(14389, 0.4076335414103674)
(14461, 0.40735200684372985)
(14891, 0.40711476122707146)
(15013, 0.4063718957024366)
(15091, 0.4053257150578368)
(15131, 0.40441053681252354)
(15193, 0.40251469168549014)
(15299, 0.403448307809185)
(15391, 0.40368123440397424)
(15619, 0.40363412436439045)
(15643, 0.4009679903200968)
(15773, 0.39990242942180876)
(15859, 0.3975086663365205)
(15889, 0.3974981579476187)
(15973, 0.3975174059086899)
(16061, 0.39783879916895803)
(16193, 0.39573950584753836)
(16363, 0.3958421042080151)
(16411, 0.3955761867920411)
(16649, 0.39358261195662364)
(16811, 0.39299174703801776)
(16883, 0.3948689030851676)
(16937, 0.3950331919739859)
(16987, 0.3942929382803927)
(17299, 0.396677767407402)
(17333, 0.39502956027960967)
(17389, 0.39383103156708005)
(17483, 0.3950338679186054)
(17573, 0.3939203470140491)
(17681, 0.3933128205128205)
(17827, 0.3948273434828952)
(17989, 0.3953868163456336)
(18097, 0.3954312554764022)
(18149, 0.3951813030706753)
(18269, 0.3953909522347138)
(18289, 0.3933259330335585)
(18379, 0.39377327490084385)
(18523, 0.3947600526750988)
(18859, 0.39395952792694394)
(19079, 0.39354157961346875)
(19211, 0.3945331001837989)
(19403, 0.3955557252491314)
(19583, 0.3954627616378068)
(19927, 0.39623114853369473)
        };  
   \end{axis}
\end{tikzpicture}
}
 \label{d1port}
\end{minipage}
\end{figure}

At the least, we expect that almost all nontrivial zeroes are accounted for by
small degree forms.  Further, we expect the number of forms of small degree in 
a given squarefree level $N$ to be $O(N^\eps)$ for any $\eps > 0$.  (In fact,
it would not be that surprising if the number of forms of small degree of a given prime level $N$
is $O(1)$---e.g., see the data in \cite{me:maeda}.)  These two expectations
suggest \cref{conj:nnontriv}.

\begin{rem}
One might further ask if the number of nontrivial zeroes for $S(\calO)$ is
actually $o(N)$.  This is certainly not clear from the limited amount of data we have,
but we suspect this is possible for the following reasons.  First,
we expect 100\% of nontrivial zeroes to come from small degree forms.
Second, if $\phi$ is degree $d$, we expect the proportion of $\phi(x_i)$'s
which vanish nontrivially can be modeled with a close-to-normal $d$-dimensional discrete 
distribution with 
variance depending on some radius $R = \langle \phi, \phi \rangle_\Lambda^{1/2}$.  
It appears that, for fixed $d$, $\langle \phi, \phi \rangle_\Lambda^{1/2}$ 
tends to increase with $N$.  If the rate of increase is strictly faster than linear,
then as $N$ increases the proportion of $\phi(x_i)$'s which are nontrivially 0 should
tend to 0, which would suggest the total number of nontrivial zeroes for 
$S(\calO_N)$ (and thus also for a single form in $S(\calO_N)$) is
$o(N)$.  In \cref{d1port}, we graph the proportion of nontrivial zeroes among 
values of degree 1 forms (excluding trivial zeroes) of prime levels $\le X$.
We suspect that this proportion may be decreasing very slowly, similar to 
how the average rank of elliptic curves appears to converge very slowly.
Indeed, using the proportion of zeroes of a form as a proxy for the probability of
vanishing $L$-values as mentioned in the introduction, 
the expectation that 0\% of elliptic curves of should have rank $\ge 2$ also suggests 
that the proportion of values which are nontrivial zeroes over all degree 1 forms 
should be 0. 
\end{rem}

\subsection{Values of quaternionic modular forms} \label{sec:vals}
One phenomenon we observed is that the values of degree 1 (primitive integral) eigenforms
tend to be quite small.
Of the 529 normalized degree 1 eigenforms of prime level $N < 20000$, only 55 take
on a value of size $\ge 5$.  Of these, only 6 have absolute values $\ge 10$---there are 
2 forms with
maximum absolute value 11 (in levels 9473 and 16193) and 4 with maximum absolute
value 12 (2 in level 8747 and 1 each in levels 13723 and 17333).  Of these 6 forms, only the one in level 13723 has root number $-1$.  In \cref{fig:deg1vals}, we present a histogram
of the values of the unique-up-to-scalars degree 1 eigenform of level 17333.  The
general shape of the histogram seems to be typical for degree 1 eigenforms---the
values tend to be small and cluster roughly symmetrically around 0 following a roughly
normal-shaped distribution.  (The histogram will
be completely symmetric in the case of root number $-1$.)

\begin{figure}
\begin{minipage}{.5\textwidth}
\captionof{figure}{Histogram of values of the degree 1 form of prime level 17333}
\resizebox{.9\linewidth}{!}{
\begin{tikzpicture}
\begin{axis}[
    ymin=0, ymax=275,
    area style,
    xtick = {-12,-9,-6,-3,0,3,6,9,12}
    ]
\addplot+[ybar interval,mark=no] plot coordinates { ( -12.5000000000000 , 1 )
( -11.5000000000000 , 0 )
( -10.5000000000000 , 0 )
( -9.50000000000000 , 0 )
( -8.50000000000000 , 2 )
( -7.50000000000000 , 0 )
( -6.50000000000000 , 5 )
( -5.50000000000000 , 20 )
( -4.50000000000000 , 61 )
( -3.50000000000000 , 98 )
( -2.50000000000000 , 196 )
( -1.50000000000000 , 214 )
( -0.500000000000000 , 262 )
( 0.500000000000000 , 200 )
( 1.50000000000000 , 185 )
( 2.50000000000000 , 114 )
( 3.50000000000000 , 61 )
( 4.50000000000000 , 20 )
( 5.50000000000000 , 5 )
( 6.50000000000000 , 0 )
( 7.50000000000000 , 1 )
( 8.50000000000000 , 0 )
( 9.50000000000000 , 0 )
( 10.5000000000000 , 0 )
( 11.5000000000000 , 0 ) (12.5, 0) };
\end{axis}
\end{tikzpicture} 
}
\label{fig:deg1vals}
\end{minipage}%
\begin{minipage}{.5\textwidth}
\captionof{figure}{Histogram of values for a degree 2 form of prime level 16889}
\resizebox{.9\linewidth}{!}{
\begin{tikzpicture}
\begin{axis}[
    view = {130}{45},
    xmin = -7.5,
    ymin = -7.5, 
    xmax = 7.5,
    ymax = 7.5,
    zmin = 0,
    unbounded coords = jump,
    colormap={pos}{color(0cm)=(white); color(6cm)=(blue)}
    ]
\addplot3+[only marks, scatter, mark=cube*,mark size = 6] table {qmf-zeroes-d2hist.dat};
\end{axis}
\end{tikzpicture} 
}
\label{fig:deg2vals}
\end{minipage}
\end{figure}

Consider a degree $d$ eigenform $\phi$, normalized so its values
are in $K=K_\phi$, as a lattice point in $K^h \hookrightarrow (\R^d)^h$
as explained above.
We expect that the $(d+1)$-dimensional histograms of values of degree $d$
eigenforms will have an analogous shape of clustering about $0$ and then rapidly
tapering off, similar to a $d$-dimensional normal distribution.
In \cref{fig:deg2vals}, we plot a histogram of the values the unique (up to
scaling and Galois conjugation) degree 2 form of prime level 16889, which has
has root number $-1$, $h=1408$ total values, and 182 zeroes (72 trivial, 110 nontrivial).  
The most common nonzero value occurs 65 times.
So even if we exclude trivial zeroes, the distribution of values
looks similar to a 2-dimensional normal
distribution.


\subsection{Zerofree forms}
In \cref{fig:zerofree,fig:zerofree2}, we plot the proportion of
zero-free forms among eigenforms with no trivial zeroes in levels $N \le X$
for $X < 4000$.
The former plot restricts to $N$ prime and the latter plot restricts to $N$ non-prime.
These data indeed suggest that this proportion tends to 100\%, in line with what
we expect from \cref{conj11}(ii).  Note that convergence appears to be slower
in the non-prime level case, which is to be expected in light of \cref{conj12}
as there tend to be more small Galois orbits when there are more Atkin--Lehner
eigenspaces (cf.\ \cite{me:maeda}).

\begin{figure}
\begin{minipage}{.5\textwidth}
\captionof{figure}{Proportion of zero-free forms for prime level}
\resizebox{.9\linewidth}{!}{
\begin{tikzpicture}
   \begin{axis}[xlabel=$X$, ymin = 0.9, ymax = 1.01, scaled x ticks = false, scaled y ticks = false]
    \addplot[only marks, mark size = 0.7] coordinates {
(13, 1.0)
(19, 1.0)
(29, 1.0)
(37, 1.0)
(43, 1.0)
(53, 1.0)
(61, 1.0)
(71, 0.9743589743589743)
(79, 0.9787234042553191)
(89, 0.9830508474576272)
(101, 0.9857142857142858)
(107, 0.9879518072289156)
(113, 0.9787234042553191)
(131, 0.9819819819819819)
(139, 0.9761904761904762)
(151, 0.9583333333333334)
(163, 0.9620253164556962)
(173, 0.9666666666666667)
(181, 0.9651741293532339)
(193, 0.968609865470852)
(199, 0.963265306122449)
(223, 0.9664179104477612)
(229, 0.9556313993174061)
(239, 0.9596273291925466)
(251, 0.9629629629629629)
(263, 0.9659685863874345)
(271, 0.9685990338164251)
(281, 0.9638009049773756)
(293, 0.9661016949152542)
(311, 0.9607072691552063)
(317, 0.9591078066914498)
(337, 0.961335676625659)
(349, 0.9636363636363636)
(359, 0.9659969088098919)
(373, 0.9677891654465594)
(383, 0.9696551724137931)
(397, 0.9672346002621232)
(409, 0.9689054726368159)
(421, 0.9705535924617197)
(433, 0.9632107023411371)
(443, 0.964021164021164)
(457, 0.9655870445344129)
(463, 0.9671814671814671)
(479, 0.9689213893967094)
(491, 0.9659685863874345)
(503, 0.9633333333333334)
(521, 0.9649960222752586)
(541, 0.9663351185921959)
(557, 0.9668874172185431)
(569, 0.9676511954992968)
(577, 0.9586440677966102)
(593, 0.958984375)
(601, 0.9606741573033708)
(613, 0.9620481927710843)
(619, 0.9633294528521537)
(641, 0.9646662927650028)
(647, 0.9654054054054054)
(659, 0.9661105318039624)
(673, 0.965133906013138)
(683, 0.9653150952613581)
(701, 0.9659896079357582)
(719, 0.9667122663018696)
(733, 0.9672855879752431)
(743, 0.9679213002566296)
(757, 0.9688667496886675)
(769, 0.9698310539018503)
(787, 0.970714564623194)
(809, 0.9708554125662378)
(821, 0.9717223650385605)
(827, 0.9725196288365453)
(839, 0.9733748271092669)
(857, 0.9740915208613729)
(863, 0.9748366013071895)
(881, 0.9755089058524173)
(887, 0.9761904761904762)
(911, 0.9768560264502555)
(929, 0.977491961414791)
(941, 0.9780876494023905)
(953, 0.9786407766990292)
(971, 0.9792172739541161)
(983, 0.979758149316509)
(997, 0.9797279958942776)
(1013, 0.9802054622901528)
(1021, 0.980698753970193)
(1033, 0.9811815150071462)
(1049, 0.9816535067347887)
(1061, 0.9811534968210718)
(1069, 0.9811446317657497)
(1091, 0.9815977484303962)
(1097, 0.9813559322033898)
(1109, 0.981799379524302)
(1123, 0.9821681864235056)
(1151, 0.9825846032060163)
(1163, 0.9829424307036247)
(1181, 0.9829351535836177)
(1193, 0.9821892393320965)
(1213, 0.9825295723384896)
(1223, 0.9813200498132005)
(1231, 0.9817232375979112)
(1249, 0.9820757937862752)
(1277, 0.9821249582358837)
(1283, 0.9821720641151456)
(1291, 0.9825264507855082)
(1301, 0.9827044025157232)
(1307, 0.9830351634793337)
(1321, 0.9815737803957106)
(1361, 0.9816405093278058)
(1373, 0.981841952353283)
(1399, 0.9821886577372471)
(1423, 0.9822327923894796)
(1429, 0.9825693110074115)
(1439, 0.9826449616574734)
(1451, 0.98297479213409)
(1459, 0.9832706523148749)
(1481, 0.9831996945399007)
(1487, 0.9833853841349157)
(1493, 0.9836569181617105)
(1511, 0.9839759036144579)
(1531, 0.9841307437233539)
(1549, 0.984391380314502)
(1559, 0.9846839638815865)
(1571, 0.9846119285634056)
(1583, 0.9848769179821172)
(1601, 0.9851216333622936)
(1609, 0.985364811451768)
(1619, 0.9853991596638656)
(1627, 0.9854112778065184)
(1657, 0.9856327695129407)
(1667, 0.9858618269327183)
(1693, 0.9857749678948928)
(1699, 0.9859935803910126)
(1721, 0.9862148190695003)
(1733, 0.9864176570458404)
(1747, 0.9866133680394162)
(1759, 0.9868131868131869)
(1783, 0.9870047829618266)
(1789, 0.9871908913004803)
(1811, 0.987391646966115)
(1831, 0.9875904860392968)
(1861, 0.987784187309128)
(1871, 0.9878863826232247)
(1877, 0.9879696769940672)
(1889, 0.9881522356569017)
(1907, 0.9880838131797824)
(1931, 0.9880248956117546)
(1949, 0.988116504854369)
(1973, 0.9882884262094305)
(1987, 0.9883843717001056)
(1997, 0.9885408140486643)
(2003, 0.98869890658252)
(2017, 0.9888430051438093)
(2029, 0.9887776983559685)
(2053, 0.9883688143239814)
(2069, 0.9883939120161235)
(2083, 0.9884111636837414)
(2089, 0.9882305194805194)
(2111, 0.988401546460472)
(2129, 0.9885488647581441)
(2137, 0.9886924876527164)
(2143, 0.9887755756526201)
(2161, 0.988914227796782)
(2203, 0.9889229613868202)
(2213, 0.9890069169960475)
(2237, 0.9890773736880645)
(2243, 0.9892207635794291)
(2267, 0.9892927250014871)
(2273, 0.9894229639205547)
(2287, 0.9895506792058516)
(2297, 0.9895576338286763)
(2311, 0.9896889694634865)
(2339, 0.9894807520143241)
(2347, 0.989379356123465)
(2357, 0.9891891891891892)
(2377, 0.9893134715025906)
(2383, 0.9893327644141021)
(2393, 0.9892461781760675)
(2411, 0.9893871605452086)
(2423, 0.9895104895104895)
(2441, 0.9896304579881056)
(2459, 0.9897559505875264)
(2473, 0.9898658718330849)
(2503, 0.9899783847514246)
(2531, 0.9900922778047596)
(2543, 0.9902059628402707)
(2551, 0.9903216623968119)
(2579, 0.9904297241508726)
(2593, 0.9903547414792487)
(2617, 0.9902796882164145)
(2633, 0.9903859235408825)
(2657, 0.9904873014448533)
(2663, 0.9905949159309703)
(2677, 0.9906492822336362)
(2687, 0.9907519972212574)
(2693, 0.9908469769240685)
(2707, 0.9908576774248417)
(2713, 0.990956507108606)
(2729, 0.9910565723793677)
(2741, 0.9911493495801087)
(2753, 0.9912377226229775)
(2777, 0.9913260983580102)
(2791, 0.991419905818501)
(2801, 0.9914654865857995)
(2819, 0.991553261379634)
(2837, 0.9915976148067839)
(2851, 0.9915286721864459)
(2861, 0.9915357169968876)
(2887, 0.9916247277097574)
(2903, 0.991711577773648)
(2917, 0.9917553093599323)
(2939, 0.9918053683942164)
(2957, 0.9918473359546914)
(2969, 0.9918551066338013)
(2999, 0.9919400452488688)
(3011, 0.9920190422850742)
(3023, 0.9920274532912753)
(3041, 0.9921032754240198)
(3061, 0.9921750076548839)
(3079, 0.9922493681550126)
(3089, 0.9922547906790412)
(3119, 0.992331592516692)
(3137, 0.9924026590693258)
(3167, 0.9924760823739257)
(3181, 0.992512853470437)
(3191, 0.9925857570164832)
(3209, 0.9926236484569555)
(3221, 0.9926895560623574)
(3251, 0.9927592288888202)
(3257, 0.9927922954238744)
(3271, 0.9927348998388911)
(3301, 0.9928025055712822)
(3313, 0.9928048962531721)
(3323, 0.9928715097018457)
(3331, 0.9929069699278973)
(3347, 0.9929683867968386)
(3361, 0.9930329638692961)
(3373, 0.9930369271160322)
(3391, 0.9930983764213385)
(3413, 0.9931606682363494)
(3449, 0.9932207157146032)
(3461, 0.9932799030543391)
(3467, 0.9932559726962458)
(3491, 0.9933142052836725)
(3511, 0.9933719744539258)
(3527, 0.99337730730358)
(3533, 0.9934326783594883)
(3541, 0.9934886639993724)
(3557, 0.993489482011776)
(3571, 0.99354523478887)
(3583, 0.9935475643968376)
(3607, 0.9935994737907307)
(3617, 0.9936510326482472)
(3631, 0.9936806488530626)
(3643, 0.9937299432238953)
(3671, 0.9937873006555131)
(3677, 0.9938355499466072)
(3697, 0.9938836447697939)
(3709, 0.9938855450463361)
(3727, 0.9939373845498034)
(3739, 0.9939836901600433)
(3767, 0.9940330513017738)
(3779, 0.9940590397373957)
(3797, 0.9941050989747001)
(3821, 0.9941072078631235)
(3833, 0.9941528389208715)
(3851, 0.9941994580188573)
(3863, 0.9942454674724494)
(3881, 0.9942894939918422)
(3907, 0.994309975051429)
(3917, 0.9943131253934145)
(3923, 0.9943581903141756)
(3931, 0.9944014701482969)
(3947, 0.9943591483586394)
(3989, 0.9943821406779304)    
    };
\end{axis}
\end{tikzpicture} 
}
\label{fig:zerofree}
\end{minipage}%
\begin{minipage}{.5\textwidth}
\captionof{figure}{Proportion of zero-free forms for non-prime level}
\resizebox{.9\linewidth}{!}{
\begin{tikzpicture}
   \begin{axis}[xlabel=$X$, ymin = 0.9, ymax = 1.01, scaled x ticks = false, scaled y ticks = false]
    \addplot[only marks, mark size = 0.7] coordinates {
(42, 1.0)
(78, 1.0)
(110, 1.0)
(138, 1.0)
(170, 1.0)
(186, 1.0)
(222, 0.9767441860465116)
(238, 0.9838709677419355)
(258, 0.9726027397260274)
(282, 0.9764705882352941)
(290, 0.9693877551020408)
(322, 0.9541284403669725)
(357, 0.9603174603174603)
(374, 0.9420289855072463)
(402, 0.9477124183006536)
(418, 0.9390243902439024)
(430, 0.9444444444444444)
(438, 0.9427083333333334)
(465, 0.9488372093023256)
(483, 0.9426229508196722)
(506, 0.943609022556391)
(534, 0.9464285714285714)
(574, 0.9446254071661238)
(595, 0.9424242424242424)
(606, 0.9454022988505747)
(615, 0.9459459459459459)
(638, 0.945679012345679)
(646, 0.9411764705882353)
(658, 0.9439461883408071)
(670, 0.9431578947368421)
(705, 0.9444444444444444)
(730, 0.943609022556391)
(754, 0.945750452079566)
(777, 0.9411764705882353)
(790, 0.9433333333333334)
(806, 0.9422776911076443)
(826, 0.9430284857571214)
(854, 0.9382183908045977)
(885, 0.9413369713506139)
(897, 0.9392338177014531)
(906, 0.9387755102040817)
(938, 0.9399038461538461)
(957, 0.9423298731257209)
(970, 0.9418994413407821)
(987, 0.9434167573449401)
(1002, 0.944560669456067)
(1015, 0.9469469469469469)
(1030, 0.9457889641819942)
(1045, 0.9476635514018692)
(1066, 0.9491371480472298)
(1085, 0.948943661971831)
(1095, 0.9490662139219015)
(1106, 0.9487603305785124)
(1130, 0.9480415667466027)
(1158, 0.9485981308411215)
(1173, 0.9500372856077554)
(1185, 0.9485507246376812)
(1221, 0.9477769936485533)
(1239, 0.9483344663494222)
(1258, 0.948582729070534)
(1270, 0.9485530546623794)
(1298, 0.9495641344956414)
(1311, 0.9523248969982343)
(1338, 0.9533141210374639)
(1358, 0.9520089285714286)
(1374, 0.952354874041621)
(1394, 0.9533762057877814)
(1407, 0.9544264012572027)
(1426, 0.9561491935483871)
(1442, 0.9554013875123885)
(1455, 0.9540396710208031)
(1474, 0.9548049476688868)
(1491, 0.956581352833638)
(1505, 0.9556933983163491)
(1515, 0.9554779572239197)
(1534, 0.955945252352438)
(1547, 0.9569397993311036)
(1562, 0.957769577695777)
(1581, 0.9565741857659831)
(1595, 0.9559228650137741)
(1605, 0.9551552318896129)
(1615, 0.9559819413092551)
(1634, 0.9563448020717721)
(1653, 0.9576149425287356)
(1670, 0.9565831274267561)
(1695, 0.9566724436741768)
(1705, 0.9573378839590444)
(1738, 0.957817207900904)
(1749, 0.9585934932632271)
(1771, 0.9590058102001291)
(1790, 0.9591384225530567)
(1810, 0.9597378277153558)
(1834, 0.9577723378212974)
(1855, 0.9586454899610428)
(1878, 0.958653278204371)
(1887, 0.9582366589327146)
(1905, 0.958974358974359)
(1930, 0.9591088550479413)
(1955, 0.9600440892807937)
(1970, 0.9602612955906369)
(1990, 0.9605157131345688)
(2013, 0.9611419508326725)
(2022, 0.9606523427388041)
(2054, 0.9612541422380831)
(2067, 0.9610778443113772)
(2085, 0.9601279842480925)
(2093, 0.9588249458222972)
(2110, 0.9559137236311922)
(2121, 0.9549170754496613)
(2139, 0.9547970479704797)
(2158, 0.9553327256153145)
(2185, 0.9552406657669815)
(2202, 0.9555801104972376)
(2230, 0.9558407657167718)
(2238, 0.9562513403388376)
(2255, 0.9573756790639365)
(2266, 0.9560845475066694)
(2278, 0.9563451776649746)
(2290, 0.9567839195979899)
(2301, 0.957089183310263)
(2318, 0.9574926542605289)
(2337, 0.9573606019679722)
(2345, 0.9574387231616949)
(2365, 0.957983193277311)
(2379, 0.9579276134484659)
(2390, 0.9586776859504132)
(2405, 0.9590207557211282)
(2410, 0.9593053850201719)
(2431, 0.9590673575129534)
(2451, 0.9591280653950953)
(2482, 0.959493670886076)
(2494, 0.9595741141241058)
(2510, 0.9596589045588717)
(2526, 0.9600324939073923)
(2546, 0.9596398134748352)
(2570, 0.9599620193068523)
(2595, 0.95984375)
(2613, 0.9601043424888752)
(2634, 0.9599393019726858)
(2658, 0.960422556167237)
(2667, 0.9601704879482658)
(2679, 0.9607730025959043)
(2690, 0.9610222729868646)
(2698, 0.9614461234288942)
(2710, 0.9613553626435172)
(2717, 0.9619572708476912)
(2737, 0.962166575939031)
(2751, 0.9621979208856487)
(2765, 0.9618401808270177)
(2770, 0.9620036813042335)
(2782, 0.9617237338888166)
(2795, 0.9621718991250643)
(2810, 0.9623676489348131)
(2829, 0.9624353964452288)
(2847, 0.962657674534782)
(2874, 0.9629721362229102)
(2895, 0.9634340222575517)
(2915, 0.9636341669687085)
(2930, 0.9638236703402012)
(2945, 0.9641330166270784)
(2955, 0.96435713445477)
(2994, 0.9642980935875216)
(3014, 0.9645642201834862)
(3026, 0.9649778986739205)
(3055, 0.9650811415780638)
(3070, 0.9652954873045793)
(3082, 0.9653846153846154)
(3115, 0.9657906168549087)
(3129, 0.9659666128163705)
(3145, 0.9657534246575342)
(3165, 0.9660963244613435)
(3178, 0.9663979901601591)
(3206, 0.9662211169826961)
(3243, 0.9663564781675018)
(3262, 0.9666227046768794)
(3282, 0.9665760595993154)
(3297, 0.9668124006359301)
(3310, 0.9666896619690549)
(3333, 0.9669687225957322)
(3345, 0.9672619047619048)
(3358, 0.9675577965940443)
(3370, 0.9674168942461625)
(3382, 0.9675609984107694)
(3405, 0.967997786590427)
(3417, 0.9679334916864608)
(3426, 0.967470097861544)
(3441, 0.967773788150808)
(3454, 0.9680766161213089)
(3471, 0.9682288923995085)
(3485, 0.9684641519296106)
(3502, 0.9685420447670902)
(3515, 0.9687339825730394)
(3530, 0.9688137321549966)
(3538, 0.9688686408504176)
(3562, 0.9690747569560845)
(3570, 0.9693928334439283)
(3590, 0.9692421868557763)
(3598, 0.9688470973017171)
(3614, 0.9691722631280848)
(3621, 0.9694492937628948)
(3638, 0.9694560340077147)
(3657, 0.9697605721393034)
(3674, 0.9696548529071114)
(3685, 0.9697132341665389)
(3702, 0.9699202430687429)
(3723, 0.9698454579721071)
(3731, 0.9700612214424369)
(3759, 0.9703660236191594)
(3782, 0.9704662835527752)
(3790, 0.9701320251064137)
(3806, 0.9702133753401117)
(3818, 0.9702991301888126)
(3835, 0.970464725643897)
(3846, 0.9704430071632242)
(3857, 0.970519516217702)
(3878, 0.9705922582410067)
(3890, 0.9707061782097465)
(3913, 0.9709794437726723)
(3922, 0.9709496460531588)
(3938, 0.971136619998673)
(3954, 0.9705360078921408)
(3965, 0.9701686966716603)
(3982, 0.9703027950310559)
(3999, 0.9703408872054953)
    };
\end{axis}
\end{tikzpicture} 
}
\label{fig:zerofree2}
\end{minipage}%
\end{figure}


\section{Periods and $L$-values} \label{sec:Lvals}


Consider a quadratic field $K=\Q(\sqrt{-D})$ of discriminant $-D$.
Identify the ideal class group $\Cl(K)$ with $K^\times \bs \hat K^\times / \hat \frako_K^\times$.
It is well known that 
$K$ embeds in $B$ if and only if $K$ is imaginary quadratic and
$K_p$ is non-split (i.e., ${-D \leg p} \ne 1$) for all $p | N$.
Assume $K$ embeds into $B$. 

For an ideal $\calJ$ in $K$ (resp.\ $B$) or an idele $\alpha$ in $\hat K^\times$
(resp.\ $\hat B^\times$), we
sometimes denote the corresponding $\frako_K$- (resp.\ $\calO$-)ideal class by $[\calJ]$ or $[\alpha]$.

\subsection{The ideal class map} \label{sec61}

Let $\calO$ be a maximal order in $B$ such that $\frako_K$ embeds into
$\calO$.  There is always some such $\calO$.
Denote by $\iota: K \to B$ an embedding such that 
$\iota(\frako_K) \subset \calO$.\footnote{Eichler showed that, summing over
maximal ideals $\calO$, the total number of such embeddings up to conjugation
by units is essentially $h_K$.  So in general, there are many choices for
such an embedding.  Moreover $\frako_K$ embeds in every maximal $\calO$ when
$h_K$ is sufficiently large relative to $N$.}
Then this embedding induces a map (necessarily not injective if $h_K > h$)
of ideal classes
$\iota_*: \Cl(K) \to \Cl(\calO)$ given by $\iota_*([t]) =  [\iota(t)]$ for  $t \in \hat K^\times$. 
This is well defined as 
$\iota(K^\times t \hat \frako_K^\times) \subset B^\times \iota(t) \hat \calO^\times$.
In what follows, we fix our embedding $\iota : K \to B$, and will simply write
$t$ for $\iota(t)$, where $t \in \hat K^\times$.  

The precise behavior of these ideal class maps is mysterious and connected
to deep problems in arithmetic, such as values of $L$-functions.
However, we can identify a few elementary properties of these maps.
First, note that $\iota_*$ is not just a map of sets, but a map of pointed sets,
where the distinguished points for $\Cl(K)$ and $\Cl(\calO)$ are the ideal
classes of $\frako_K$ and $\calO$, respectively.  In particular, if $K$ has class
number 1, then $\iota_*$ does not depend upon the choice of $\iota$ (and in this
case, there is only one choice for $\calO$ up to isomorphism).

\begin{rem} \label{rem:61}
If $K=\Q(i)$ or $\Q(\sqrt{-3})$, then $\frako_K^\times$ has order 4 or 6.
If $N > 2$ is prime and $K$ embeds in $B$, then there is a unique $i$ such that 
$[\frako_K^\times: \Z^\times]$ divides $e_i$ (e.g., see \cite[Table 1.3]{gross}).  Consequently, there is a unique maximal order $\calO \subset B$
up to isomorphism such that $\calO^\times$ contains a 4th or 6th root of unity.
This provides one way to specify exactly which $\calO$ possess
an embedding of $\frako_K$ in these
special cases.
\end{rem}

Next we give two slightly more interesting properties.
For $d | N$, write $\sigma_d(x) = \prod_{p | d} \sigma_p(x)$.  

\begin{prop} \label{prop:classmap1}
The class map satisfies $\iota_*(t^{-1}) = \sigma_N(\iota_*(t))$ for
$t \in \Cl(K)$.
\end{prop}

\begin{proof}
Consider the orthogonal decomposition (with respect to the norm form)
$B = K \oplus j K$, where $j \in B^\times$.  Let $\varpi_N \in B^\times$ be an element of reduced norm $N$.
We claim there is a rational element $\alpha \in jK \cap
\varpi_N \hat \calO^\times$.  Assuming there is such an $\alpha$, note conjugation by
$\alpha$ acts on $K$ as the standard (Galois) involution $a \mapsto \bar a$ since
$\alpha \in jK$.
Then for $t \in \hat K^\times$, we have
\[ B^\times \bar t \hat \calO^\times = B^\times \alpha^{-1} t \alpha \hat \calO^\times
= B^\times t \varpi_N \hat \calO^\times = \sigma_N(B^\times t \hat \calO^\times). \]
Since $[ \hspace{1.5pt} \bar t \hspace{1.5pt} ] = [t^{-1}]$ in $\Cl(K)$,
we get that $\iota_*([t^{-1}]) = \sigma_N(\iota_*([t]))$.  

So it suffices to prove the existence of an $\alpha$ as above.  For this, we
claim that there exists a prime $q \nmid DN$ split in $K$
such that $B \simeq {-D, -Nq \leg \Q }$.  That is, we want to show
the existence of a prime $q \nmid DN$ split in $K$ such that ${-D, -Nq \leg \Q}$
has discriminant $N$, i.e., such that ${-D, -Nq \leg \Q_p}$ is a division algebra if and only if
$p | N$.  In fact, since ${-D, -Nq \leg \Q}$ is necessarily definite, it
must be ramified at an odd number of finite primes, so 
we may restrict this latter condition to odd primes $p$.

Standard splitting criteria for quaternion algebras imply the following result.

\begin{lem} Suppose $p$ is odd.  Assume $a, b \in \Z$ are nonzero
and squarefree with $v_p(a) \le v_p(b)$.  Then ${a, b \leg \Q_p}$ is a
division algebra if and only if (1) $p \nmid a$, $p | b$ and ${a \leg p} = -1$;
or (2) $p | a$ and $-b/a$ is a nonsquare mod $p$.
\end{lem}

Assume $q$ is a prime not dividing $DN$.
First, we want $q$ to satisfy (i) ${-D \leg q} = +1$ for $q$ to be split in $K$, 
and by the above lemma, this also means that ${-D, -Nq \leg \Q}$
is split at $q$---i.e., the localization at $q$ is isomorphic to
$M_2(\Q_q)$---as desired.  Also, if $p \nmid DNq$ is odd, then 
${-D, -Nq \leg \Q}$ is split at $p$  by the lemma.

Next consider a prime $p$ such that $p | N$ but $p \nmid D$. 
Then ${-D \leg p} = -1$ since $K$ embeds in $B$, which means that 
${-D, -Nq \leg \Q}$ is divison at $p$, also as desired.

Now consider an odd prime $p$ such that $p | D$ but $p \nmid N$.
Then ${-D, -Nq \leg \Q}$ is split at $p$ if and only if (ii) ${q \leg p} = {-N \leg p}$.

Finally, consider an odd prime $p | \gcd(D,N)$.  Then ${-D, -Nq \leg \Q}$
is division at $p$ if and only if (iii) ${cq \leg p} = -1$ where $c$ is an integer equivalent
to $N/D \mod p$.  

By Dirichlet's theorem on primes in progressions, there exists a prime $q$
satisfying (i), (ii) and (iii) as above, which then gives us a $q$ as claimed.  Consequently,
we may take our orthogonal decomposition $B = K \oplus j K$ above with
$j^2 = -Nq$, and then we may take $\alpha = ja$ where $a \in K$ is an element of
norm $q^{-1}$.  
\end{proof}

\begin{cor} \label{cor:fixed-pt-1cpg}
Suppose $t \in \Cl(K)$ has order $\le 2$.  Then $\iota_*(t)$ is
a fixed point of $\sigma_N$.
\end{cor}

Note that this gives an alternative proof of \cref{sigmaN-fixpt}.
Here is another result in the spirit of this corollary.

\begin{prop} \label{prop:classmap2}
Let $d = D$ if $4 \nmid D$ and $d = \frac D4$ if
$4 | D$.  Suppose $d | N$.  Then for $t \in \Cl(K)$, 
$\sigma_d(\iota_*(t)) = \iota_*(t)$, i.e., the image of $\iota_*$ lies in 
the set of fixed points of $\sigma_d$.
\end{prop}

\begin{proof}
Let $\alpha = \iota(\sqrt{-d})$.  Then $\alpha_p \in \calO_p^\times$
for $p \nmid d$ and $\alpha_p$ is a uniformizer $\varpi_{B_p}$ in $B_p$ for $p | d$.
Thus for $t \in \hat K^\times$, $t$ commutes with $\alpha \in K$, so
\[ B^\times t \hat \calO^\times = B^\times \alpha^{-1} t \alpha \hat \calO^\times
= B^\times t \prod_{p | d} \varpi_{B_p} \hat \calO^\times, \]
which proves the assertion.
\end{proof}

We remark that, when $d=p$, the conclusion that $\sigma_p$ has fixed points in this situation provides a different proof of part of \cref{lem:fixed-pt}.

Here is one more elementary result in the case that $K$ and $B$ have joint
ramification.

\begin{prop} \label{prop:classmap3}
 Suppose $p | \gcd(D,N)$.  Let $\calJ_p$ denote the class of
an ideal in $\frako_K$ of norm $p$.  Multiplication by $\calJ_p$ acts on $\Cl(K)$ 
as an involution, which we also denote by $\sigma_p$.  
Then $\iota_*: \Cl(K) \to \Cl(\calO)$ respects the action of $\sigma_p$, i.e.,
$\iota_*(\sigma_p(t)) = \iota_*(t \calJ_p) = \sigma_p(\iota_*(t))$.
\end{prop}

\begin{proof} This is clear, as a uniformizer for $K_p$ is also a uniformizer for
$B_p$.
\end{proof}

Other properties of the ideal class map require much deeper understanding
of arithmetic.  Notably, Michel uses subconvexity of $L$-functions
to show that the ideal class maps are equidistributed in a suitable sense
(for fixed $B$ and varying $K$, as $D \to \infty$) \cite[Theorem 10]{michel}.
One application of equidistribution of the ideal class map is to the existence
of nonvanishing twisted central $L$-values \cite{michel-venkatesh}---see also
\cref{rem614}.

\begin{rem} Here is an unexpected (at least to us) 
consequence of equidistribution combined with
\cref{cor:fixed-pt-1cpg}.  
The latter says that if $K \subset B$ has one class per genus, i.e.,
$\Cl(K)$ has exponent $\le 2$, then the image of the class map $\iota_*$
is contained in the fixed points of $\sigma_N$.  However, assuming 
$\dim S_2^{-_N}(N) \ne 0$ (e.g., if $N > 71$ is prime \cite[Proposition 4.2]{me:dim}), 
then not all elements of 
$\Cl(\calO)$ are fixed by $\sigma_N$.  Thus the ideal class map is
not surjective.  But equidistribution of this map means that for $D$ sufficiently
large, $\Cl(K) \to \Cl(\calO)$ is surjective.  Consequently, there can only
be finitely many one-class-per-genus fields $K$ embedding in $B$.
Even though the much stronger result that there are at most 66 imaginary quadratic
one-class-per-genus fields is known \cite{weinberger}, 
 this argument is interesting, as to our knowledge it is rather different
than existing proofs of finiteness results for certain class group structures.
Of course, deep facts about $L$-functions (subconvexity)  still underlie
this argument.
\end{rem}

\subsection{Toric periods}
Let $\chi$ be a character of $\Cl(K)$.
We fix Haar measures on $\hat K^\times$ and $\hat \Q^\times$ such that the
quotient $K^\times \hat \Q^\times \bs \hat K^\times / \hat \frako_K^\times \simeq \Cl(K)$ has measure $h_{K}$.

Define the period integral $P_{K,\chi} : M(\calO) \to \C$ by
\[ P_{K,\chi}(\phi) = \int_{K^\times \hat \Q^\times \bs \hat K^\times} 
\phi(t) \chi^{-1}(t) \, dt = \sum_{t \in \Cl(K)} \phi(\iota_*(t)) \chi^{-1}(t). \]
In the case that $\chi$ is trivial 
we simply denote $P_{K,\chi}$ by $P_K$. 

For $\pi$ an automorphic representation occurring in $S(\calO)$ with
 corresponding cuspidal automorphic representation $\pi'$of $\GL_2(\A_\Q)$, and
$\phi \in \pi^{\hat \calO^\times}$, Waldspurger \cite{wald} proved a formula of the
form
\begin{equation}
 |P_{K,\chi}(\phi)|^2 = c(\phi, K, \chi) L(\tfrac 12, \pi'_K \otimes \chi).
\end{equation}
Here $\pi'_K$ denotes the base change of $\pi'$ to $\GL_2(\A_K)$. 
If $\chi$ is trivial, then $L(s, \pi'_K) = L(s, \pi')L(s, \pi' \otimes \eta_{D})$,
where $\eta_D$ is the quadratic Dirichlet associated to $K/\Q$.

For simplicity in what follows, for an eigenform $\phi \in \pi \cap S(\calO)$ as above,
we write $L(s, \phi) =  L(s, \pi')$, $L(s, \phi_K \otimes \chi) = L(s, \pi'_K \otimes \chi)$,
etc., and similarly for $\eps$-factors.

In particular, $P_{K,\chi}(\phi) \ne 0$ implies the twisted central
$L$-value $L(\frac 12, \phi_K \otimes \chi) \ne 0$.  The converse is not true
as $c(\phi, K, \chi)$ may be $0$.  Indeed, by Tunnell \cite{tunnell}, 
$P_{K,\chi}$ is identically $0$
on $\pi$ (the period is well defined for any smooth $\phi \in \pi$, 
not just $\phi \in S(\calO)$)
unless the local $\eps$-factors satisfy $\eps_v(\frac 12, \phi_K \otimes \chi) = +1$ 
if and only if $B_v$ is ramified.  Note that this local $\eps$-factor
condition forces the global root number of $\phi_K \otimes \chi$ to be $+1$.

Precisely, the following holds.  
For an eigenform $\phi \in M(\calO)$ and $d | N$, let
$\eps_d(\phi)$ be the $T_d$-eigenvalue of $\phi$.
For $p | D$, let $\varpi_{K_p}$ denote a uniformizer for $K_p$.

\begin{prop} \label{prop:per-van}
 Let  $\phi \in S(\calO)$ be an eigenform and $\chi$ be a character of $\Cl(K)$.  
 Then $P_{K,\chi}(\phi) = 0$ unless:
\begin{enumerate}[(i)]
\item $\eps_N(\phi) = +1$ if $\chi$ is quadratic; and

\item $\eps_p(\phi) = \chi_p(\varpi_{K_p})$ for all $p | \gcd(N,D)$.
\end{enumerate}
Moreover, assuming (ii), $L(\frac 12, \phi_K \otimes \chi) \ne 0 \iff P_{K,\chi}(\phi) \ne 0$.
\end{prop}

Condition (ii) corresponds to Tunnell's $\eps$-factor criterion.
See \cref{rem610} about (i).

\begin{proof}
First suppose $\chi$ is quadratic.  Then by \cref{prop:classmap1},
\[ \eps_N(\phi) P_{K,\chi}(\phi) = 
P_{K,\chi}(T_N(\phi)) = \sum \phi(\sigma_N(\iota_*(t))) \chi^{-1}(t) = 
\sum \phi(\iota_*(t)) \chi(t) = P_{K,\chi}(\phi). \]
In particular, $P_{K,\chi}(\phi) = 0$ unless (i) holds.

Similarly, by \cref{prop:classmap3}, one sees that 
$\eps_p(\phi) P_{K,\chi}(\phi) = \chi_p(\varpi_{K_p})P_{K,\chi}(\phi)$
for $p | \gcd(N,D)$. 
Hence $P_{K,\chi}(\phi) = 0$ unless (ii) holds.

Now we explain the final part. 
The above integral definition of $P_{K,\chi}(\psi)$ makes sense for smooth elements in 
$\psi  \in \pi$, so we may view $P_{K,\chi}$ as an element
of $\Hom_{\A_K^\times}(\pi, \chi)$.  Via the factorization of $\pi = \otimes \pi_v$,
this global Hom space can only be nonzero if the local Hom spaces
$\Hom_{K_v^\times}(\pi_v, \chi_v)$ are all nonzero.  
It follows from Jacquet's relative trace formula approach to Waldspurger's theorem
that $L(\frac 12, \pi_K \otimes \chi) \ne 0$ if and only if $P_{K,\chi}$ is not
identically zero on $\pi$ (the main result of \cite{jacquet}, with a gap fixed in \cite{jacquet-chen}---see p.\ 41 of the latter paper).  From the work of Gross and Prasad on
test vectors, $P_{K,\chi}$ is not identically 0 on $\pi$ if and only if $P_{K,\chi}(\phi) \ne 0$
if $\phi$ is a nonzero element of $\pi^{\hat \calO^\times}$ 
\cite[Proposition 2.6]{gross-prasad}.

It remains to check that (ii) guarantees each local Hom space is nonzero.
While one could use the $\eps$-factor criterion in \cite{tunnell} to check this, it
is simple to give a direct argument.
The local archimedean Hom space is nonzero since $\pi_\infty$ and $\chi_\infty$
are trivial.  For finite $p \nmid N$, $\pi_p$ is an unramified principal series and thus
the local Hom space $\Hom_{K_p^\times}(\pi_p, \mu_p) \ne 0$ for any character $\mu_p$
of $K_p$ with trivial central character.  

Now suppose $p | N$.  Then $\pi_p$ is 1-dimensional, either trivial if $\eps_p(\phi) = 1$
or $\eta_p \circ \det$ if $\eps_p(\phi) = -1$, where $\eta_p$ is the unramified quadratic
character of $\Q_p^\times$.  In either case, the local Hom space 
$\Hom_{K_p^\times}(\pi_p, \chi_p)$ being nonzero simply means that 
$\pi_p |_{K_p^\times} = \chi_p$.  Note $\pi_p |_{K_p^\times}$ and $\chi_p$ are
both characters of $K_p^\times/\frako_{K_p}^\times \Q_p^\times$.
If $K_p/\Q_p$ is unramified, then this quotient is trivial and 
$\pi_p |_{K_p^\times} = \chi_p$ holds automatically.  
If $K_p/\Q_p$ is ramified, then the quotient has order 2 and is generated
by a uniformizer $\varpi_{K_p}$, in which case (ii) is equivalent to both characters
agreeing on $K_p^\times$.
\end{proof}

\begin{rem} \label{rem610}
 While this type of result is clear to experts and has often been used,
to our knowledge this is the first direct proof for why the periods
must vanish when (i) or (ii) fails to hold.  Rather the standard argument for vanishing
of periods in such situations is either because the local Hom space is zero or
that a root number condition forces an $L$-value to vanish.  For instance,
regarding (i), when $\chi$
is quadratic, the automorphic induction to $\GL_2(\A_\Q)$ is a nontrivial isobaric sum
and so $L(s, \phi_K \otimes \chi)$ factors into 2 degree 2 $L$-functions over $\Q$.  
When $\chi = 1$, we see $L(s,\phi_K) = L(s, \phi) L(s, \phi \otimes \eta_{K/\Q})$
where $\eta_{K/\Q}$ is the quadratic character over $\Q$ associated to $K/\Q$.  
From this it is not hard to see that the two $L$-functions on the right have root number $+1$
if and only if $\phi$ does assuming (ii).   When $\chi$ is nontrivial quadratic, it is 
less obvious that the root numbers of the factored $L$-functions match the root 
number of $L(s,\phi)$.  We expect one can check that this is so, but have not
attempted to do so and have not seen this in the literature.
Accordingly, while the $\chi=1$ case of (i) is well known, we are unsure whether
(ii) was previously known for $\chi$ of order 2.
\end{rem}

In the following results, recall that we assume $\frako_K \subset \calO$.

\begin{thm} \label{thm:1cpg}
Suppose $K$ has one class per genus. 
Let $\chi$ be a character of ${\Cl}(K)$, $\eps$ be a sign pattern for $N$, and $\phi \in S^\eps(\calO)$ be an eigenform.
Assume one of the following:
\begin{enumerate}
\item $h_K = 1$; or

\item for each $p | D$, we have that
$p | N$ and  $\chi_p(\varpi_{K_p}) = \eps_p$.
\end{enumerate}
Then $L(\frac 12, \phi_K \otimes \chi) \ne 0 \iff \phi(1) \ne 0$.
\end{thm}

\begin{proof} By the previous result,
it suffices to show that $P_{K,\chi}(\phi)$ is a non-zero multiple of $\phi(1)$.

Case (1) is clear, as then $x_0 := \iota_*([\frako_K])$ is a fixed point of 
$\sigma_N$ and thus $P_K(\phi) = \phi(x_0) = \phi(1)$.  (Necessarily $\chi=1$ here.)

Now assume (2).
For $p | D$, let $\hat \varpi_{K_p} \in \hat K$ denote an element whose local components 
$(\hat \varpi_{K_p})_v \in \frako_{K,v}^\times$ are integral units for all finite primes 
$v$ of $K$ not lying above $p$, and whose local component is a uniformizer at 
the prime above $p$.
Then $\Cl(K)$ is generated by the $\hat \varpi_{K_p}$'s, i.e., the ideals of
absolute norm $p$, for $p | D$.  Thus for any $t \in \Cl(K)$, we may choose a
representative in $\hat K^\times$ of the form $\prod \hat \varpi_{K_p}$ for some subset
of $\{ p : p | D \}$.  By \cref{prop:classmap3}, we see that $\phi(t) = \phi(\prod \hat \varpi_{K_p})
= \prod \eps_p(\phi) \cdot \phi(1) = \chi(t) \phi(1)$.  
Since $\Cl(K)$ has exponent 2, $\chi$ must be quadratic.  Thus we deduce 
$P_{K,\chi}(\phi) = h_K \phi(1)$.
\end{proof}

\begin{cor} Fix $\calO \subset B$, a sign pattern $\eps$ for $N$, 
and let $Z$ be the collection of pairs $(K,\chi)$
satisfying the hypotheses of the previous theorem.  Then, for any fixed eigenform 
$\phi \in S^\eps(\calO)$, the twisted central values $L(\frac 12, \phi_K \otimes \chi)$
for $(K, \chi) \in Z$ are either all zero or all nonzero.
\end{cor} 

Note that one can compute examples of $B$ such that $|Z| > 1$, so that
this corollary has some content, however we do not know if there are infintely
many such $Z$.

\begin{thm} \label{thm:existchi}
Suppose  $\phi \in S(\calO)$ is
an eigenform such that $\phi(1) \ne 0$.  
Then there exists a character $\chi$ of $\Cl(K)$ such that
$L(\frac 12, \phi_K \otimes \chi) \ne 0$.
\end{thm}

\begin{proof} 
By orthogonality of characters, we have 
$\sum_{\chi \in \widehat{\Cl}(K)} P_{K,\chi}(\phi) = h_K \phi(1)$.
Hence some $P_{K,\chi}(\phi) \ne 0$ if $\phi(1)$ is.
\end{proof}

When $h_K=1$, this theorem overlaps with the previous one.

\begin{rem} \label{rem614}
The above argument similarly works (via a different linear combination of
characters) if $\phi$ is nonzero on some element in
the image of $\Cl(K)$.  This fact was used in \cite{michel-venkatesh} to
make the same conclusion for any eigenform $\phi$ if $D$ is sufficiently large
(and in fact get a lower bound on how many $\chi$ 
satisfy $L(\frac 12, \phi_K \otimes \chi) \ne 0$).
Our point is that the nonvanishing of $\phi$ at a certain value
also implies the nonvanishing of twisted $L$-values for small $D$.
\end{rem}

We also remark that \cref{prop:classmap1} implies 
$P_{K,\chi^{-1}}(\phi) = P_{K,\chi}(\sigma_N(\phi)) = \eps_N(\phi) P_{K,\chi}(\phi)$.
This corresponds to the fact that 
$L(\frac 12, \phi_K \otimes \chi) = \eps(\frac 12, \phi_K \otimes \chi)
L(\frac 12, \phi_K \otimes \chi^{-1})$, because the representations
$\pi_K \otimes \chi$ and $\pi_K \otimes \chi^{-1}$ are contragredient.  
In particular, another statement one can make about non-vanishing of twists is that
$L(\frac 12, \phi_K \otimes \chi) \ne 0$ if and only if $L(\frac 12, \phi_K \otimes \chi^{-1}) \ne 0$.

\begin{example} We continue the example $N=154$ from 
\cref{ex:154} and the notation therein.  
Here the fields $K=\Q(\sqrt{-D})$ with class number
1 which embed in $B$ are those with $D \in \{ 4, 11, 67, 163 \}$.
Since $\phi_1, \phi_2$ are zero-free,
$L(\frac 12, \phi_{K}) = L(\frac 12, \phi) L(\frac 12, \phi \otimes \eta_D)
\ne 0$ for $\phi = \phi_1, \phi_2$ and each such $K$.  
Since $\phi_5$ has root number $-1$,
each $L(\frac 12, \phi_{5,K}) = 0$.  

Now $\phi_3$ and $\phi_4$ correspond to weight 2 rational newforms
in $S_2(154)$ with root number $+1$.  All elliptic curves of conductor 154
and root number $+1$ are rank 0, so $L(\frac 12, \phi_3)$ and $L(\frac 12, \phi_4)$
are nonzero.  Hence for $\phi \in \{ \phi_3, \phi_4 \}$, we have
$L(\frac 12, \phi \otimes \eta_D) \ne 0$ if and only if $P_K(\phi) \ne 0$.
Recall that $\sigma_2$ fixes the ideal classes in $X_3$.  
Thus \cref{prop:classmap2} implies $\Z[i]$ must embed into the
left order $\calO_\ell(x_5) = \calO_\ell(x_6)$ associated to $X_3$.  (Alternatively
one can check in Magma that $\calO_\ell(x_5)$ has 4 units and apply \cref{rem:61}.)
Similarly, one sees that $\Z[\frac {1+\sqrt{-11}}2]$ embeds in the left order 
associated to $X_1$.
Consequently, we see $L(\frac 12, \phi_3 \otimes \eta_4) \ne 0$,
$L(\frac 12, \phi_3 \otimes \eta_{11}) = 0$, 
$L(\frac 12, \phi_4 \otimes \eta_4) = 0$, and
$L(\frac 12, \phi_4 \otimes \eta_{11}) \ne 0$. 
\end{example}

We find this example interesting
because here it is \emph{trivial} zeroes that force the elliptic curves associated
to $\phi_3 \otimes \eta_{11}$ and $\phi_4 \otimes \eta_4$ to have rank at
least 2.   However, we expect that this phenomenon of
trivial zeroes forcing analytic rank $\ge 2$ is quite rare---possibly it only
happens finitely often.

%
%

\end{document}